\documentclass[a4paper,12pt,leqno]{amsart}
\usepackage{amsfonts,amssymb,amscd,amsmath,latexsym,amsbsy,enumerate,a4wide,mathrsfs}

\newtheorem{thm}{Theorem}[section]
\newtheorem{defn}[thm]{Definition}
\newtheorem{lem}[thm]{Lemma}
\newtheorem{cor}[thm]{Corollary}
\newtheorem{prop}[thm]{Proposition}

\theoremstyle{definition}
\newtheorem{rmk}[thm]{Remark}

\numberwithin{equation}{section}

\def\al{\alpha}

\def\de{\delta}
\def\la{\lambda}

\def\La{\Lambda}
\def\Ga{\Gamma}

\def\C{\mathbb{C}}
\def\N{\mathbb{N}}

\def\cR{\mathcal R}

\def\cH{\mathcal H}

\def\pol{\text{\rm pol}}
\def\Id{\text{\rm I}}
\newcommand{\rFs}[5]{\,_{#1}F_{#2} \left( \genfrac{.}{.}{0pt}{}{#3}{#4};#5 \right)}

\newcommand{\SU}{\mathrm{SU}}
\newcommand{\U}{\mathrm{U}}

\title[Matrix-valued Gegenbauer polynomials]
{Matrix-valued Gegenbauer polynomials}

\author{Erik Koelink}
\address{IMAPP, Radboud Universiteit, 
PO Box 9010, 6500 GL Nijmegen, 
the Netherlands}
\email{e.koelink@math.ru.nl}
\author{Ana M. de los R\'\i os}
\address{Departamento de An\'{a}lisis Matem\'{a}tico, Universidad de Sevilla, 
Apdo (PO Box) 1160, 41080 Sevilla, Spain}
\email{amdelosrios@us.es}
\author{Pablo Rom\'an}
\address{CIEM,
FaMAF, Universidad Nacional de C\'ordoba, Medina Allende s/n Ciudad
Universitaria, C\'ordoba, Argentina}
\email{roman@famaf.unc.edu.ar}
\date{\today}

\begin{document}

\begin{abstract}
We introduce matrix-valued weight functions of arbitrary size, which are analogues of the weight function for the Gegenbauer or ultraspherical polynomials for the parameter $\nu>0$. The LDU-decomposition of the 
weight is explicitly given in terms of Gegenbauer polynomials. We establish a matrix-valued Pearson equation for these matrix weights leading to explicit shift operators relating the weights for parameter $\nu$ and $\nu+1$. The matrix coefficients of the Pearson equation are obtained using a special matrix-valued differential operator in a commutative algebra of symmetric differential operators.

The corresponding orthogonal polynomials are the matrix-valued Gegenbauer polynomials which are eigenfunctions for the symmetric matrix-valued differential operators. Using the shift operators we find the squared norm and we establish a simple Rodrigues formula.
The three-term recurrence relation is obtained explicitly using the shift operators as well. We give an explicit non-trivial expression for the matrix entries of 
the matrix-valued Gegenbauer polynomials in terms of scalar-valued Gegenbauer and Racah polynomials using the LDU-decomposition and differential operators. The case $\nu=1$ reduces to the case of matrix-valued Chebyshev polynomials previously obtained using 
group theoretic considerations. 
\end{abstract}

\maketitle
%%%%%%%%%%%%%%%%%%%%%%%%%%%%%%%%%%%%%%%%%%%%%%%%%%%%%%%%%%%%%%%%%%%%%%
%%%%%%%%%%%%%%%%%%%%%%%%%%%%%%%%%%%%%%%%%%%%%%%%%%%%%%%%%%%%%%%%%%%%%%
%%%%%%%%%%%%%%%%%%%%%%%%%%%%%%%%%%%%%%%%%%%%%%%%%%%%%%%%%%%%%%%%%%%%%%
\section{Introduction}\label{sec:intro}

Matrix-valued orthogonal polynomials have been introduced and studied by Krein 
in connection with spectral analysis and moment problems, see \cite{Krein2}, \cite{Krein1} and  \cite{Bere}. 
Matrix-valued orthogonal polynomials have 
applications in the study of higher order recurrences and their spectral 
analysis, see e.g. \cite{ApteN}, \cite{DuravA}, \cite{GroenIK}, 
communication theory \cite{Geronimo},
Toda lattices \cite{AlvaAGAMM}, \cite{AriM}, \cite{Gek} to name a few. 

Using matrix-valued spherical functions \cite{GangV} one is able to 
associate matrix-valued orthogonal polynomials to certain 
symmetric pairs. The first case was studied by Gr\"unbaum et al. \cite{GrunPT}
for the case of the symmetric pair $(G,K)=(\SU(3), \U(2))$ relying heavily 
on invariant differential operators. Another more direct approach
was developed in \cite{KoelvPR1}, \cite{KoelvPR2} for the case
of $(\SU(2)\times \SU(2), \text{diag})$ inspired by \cite{Koor1}. See also \cite{HeckvP}, \cite{Prui} 
for  the general set-up in the context of the so-called multiplicity free pairs. 
In particular, \cite{KoelvPR1}, \cite{KoelvPR2} gives a detailed study 
of the matrix-valued orthogonal polynomials, which can be considered
as matrix-valued analogues of the Chebyshev polynomials, i.e. the 
spherical polynomials on $(\SU(2)\times \SU(2), \text{diag})$ better
known as the characters on $\SU(2)$, see also 
\cite{AldeKR} for the quantum group case. 

The purpose of this paper is to extend the family of matrix-valued Chebyshev
polynomials in \cite{KoelvPR1}, \cite{KoelvPR2} to a family of matrix-valued Gegenbauer
polynomials using shift operators, where the lowering operator is the derivative. In the matrix-valued
setting the raising operator, being the adjoint of the derivative, is harder to obtain and involves a 
matrix-valued analogue of the Pearson equation. The ingredients for the matrix-valued Pearson equation
are obtained from a matrix-valued differential operator which is well suited for a Darboux factorisation.
The use of shift operators is a well established technique in special functions, see e.g.  \cite{Koor}, \cite{Opdam} for explicit examples and \cite{AndrAR}, \cite{Isma}, \cite{KoekLS}
for general information. 

Let us recall the classical case of the shift operators for the Gegenbauer polynomials. The Chebyshev polynomials can be considered as the Gegenbauer  polynomials $C^{(\nu)}_n$ with $\nu=1$ and by successive derivation
the Gegenbauer polynomials $C^{(\nu)}_n$ with $\nu\in \N\setminus\{ 0\}$
can be obtained. Moreover, many properties of the Chebyshev polynomials
can be transported to Gegenbauer polynomials $C^{(\nu)}_n$ with integer $\nu$.
Then one can obtain the general family of Gegenbauer polynomials $C^{(\nu)}_n$
by continuation in $\nu$, or one can extend differentation using suitable 
fractional integral operators. 
Let us recall this in some more detail, see e.g. \cite{AndrAR}, \cite{Isma},  \cite{KoekLS}. 
The Gegenbauer polynomials 
\begin{equation}\label{eq:defGegenbauerpols}
C^{(\nu)}_n(x) = \frac{(2\nu)_n}{n!} \rFs{2}{1}{-n, \, n+2\nu}{\nu+\frac12}{\frac12(1-x)}
\end{equation}
are eigenfunctions to 
an explicit second-order differential operator of hypergeometric type;
\begin{equation}\label{eq:DEscalarGegenbauer}
D^{(\nu)}C^{(\nu)}_n = - n(n+2\nu) C^{(\nu)}_n, \qquad
(D^{(\nu)}y)(x)=(1-x^2) y''(x) - (2\nu+1)xy'(x).
\end{equation}
The orthogonality for the Gegenbauer polynomials is 
\begin{equation}\label{eq:orthoscalarGegenbauer}
\int_{-1}^1 C^{(\nu)}_n(x) C^{(\nu)}_m(x) w^{(\nu)}(x)\, dx = \de_{m,n} h_n^{(\nu)},
\qquad h_n^{(\nu)}=\frac{(2\nu)_n\, \sqrt{\pi}\, \Ga(\nu+\frac12)}{n!\, (n+\nu)\, \Ga(\nu)}
\end{equation}
with $w^{(\nu)}(x)=(1-x^2)^{\nu-1/2}$.
In particular, using \eqref{eq:defGegenbauerpols} shows that 
$\frac{dC^{(\nu)}_n}{dx}(x)= 2\nu\, C^{(\nu+1)}_{n-1}(x)$. 
We view the derivative $\frac{d}{dx}$ as an unbounded operator 
$L^2(w^{(\nu)})\to  L^2(w^{(\nu+1)})$ with respect to the integral over $(-1,1)$.
It is a shift (or lowering) operator since it lowers the degree of a polynomial by one. 
Its adjoint $-S^{(\nu)} \colon L^2(w^{(\nu+1)})\to  L^2(w^{(\nu)})$
is explicitly given by  
$\bigl(S^{(\nu)} f\bigr)(x) = (1-x^2)\frac{df}{dx}(x) -(2\nu+1)xf(x)$. 
This follows from the Pearson equation 
\begin{equation}\label{eq:Pearsonscalar}
w^{(\nu+1)}(x) =  (1-x^2)\, w^{(\nu)}(x), \qquad 
\frac{dw^{(\nu+1)}}{dx}(x) = -(2\nu+1)x\, w^{(\nu)}(x),
\end{equation}
see also \cite{RahmS}. It follows that  $S^{(\nu)}\colon C^{(\nu+1)}_{n-1}\mapsto 
\frac{-n(2\nu+n)}{2\nu} C^{(\nu)}_n$
by considering the orthogonality relations and the leading coefficients. 
Then  $D^{(\nu)} = S^{(\nu)}\circ \frac{d}{dx}$ gives a factorisation 
in terms of a lowering and a raising operator of the 
differential operator \eqref{eq:DEscalarGegenbauer} 
having the Gegenbauer polynomials as eigenfunctions. 
In particular, we find $\frac{d}{dx}\circ D^{(\nu)} = 
\frac{d}{dx}\circ S^{(\nu)}\circ \frac{d}{dx}$, so that $C^{(\nu+1)}_{n-1}(x)$,
being a multiple of the derivative of $C^{(\nu)}_{n}(x)$, is an eigenfunction to 
$\frac{d}{dx}\circ S^{(\nu)}$. 
This Darboux transformation, i.e. 
interchanging the raising and lowering operators, 
gives $\frac{d}{dx} \circ S^{(\nu)} = D^{(\nu+1)} -(2\nu+1)$, 
and this is also known as a transmutation property.

The main results of this paper are stated in Sections \ref{sec:weightfunctionandsymmetricdiffop} and \ref{sec:MVGegenbauerpolsandprops}. In Section \ref{sec:weightfunctionandsymmetricdiffop}  we define the weight
function and we state its LDU-decomposition which is essential in most of the proofs in this paper. The inverse of the matrix $L$ in this decomposition is a particular case of the results of Cagliero and Koornwinder \cite{CaglK}. We obtain a commutative algebra generated by two matrix-valued differential operators which are symmetric with respect to the matrix weight. From these differential operators we obtain the ingredients for a matrix-valued Pearson equation which, in turn, gives us the matrix-valued adjoint of the derivative. Finally, in Section \ref{sec:weightfunctionandsymmetricdiffop}  we describe the commutant algebra of the weight, showing that there is a non-trivial orthogonal decomposition for each weight. In \cite{KoelR} we show that no further reduction is possible. The proofs of these statements are given in Section \ref{sec:differentialoperators}.

In Section \ref{sec:MVGegenbauerpolsandprops} we discuss the corresponding monic matrix-valued orthogonal polynomials. Using the shift operators we explicitly evaluate the squared norms and we obtain a simple Rodrigues formula for the polynomials. Moreover, the polynomials are eigenfunctions for the matrix-valued differential operators in Section \ref{sec:weightfunctionandsymmetricdiffop}. The matrix entries of the monic matrix-valued orthogonal polynomials are explicitly calculated in terms of Gegenbauer and Racah polymomials. We give the three-term recurrence relation for the polynomials explicitly. We give the proofs in Section \ref{sec:MVOPs}.

In Section \ref{sec:differentialoperators} we give an elementary approach to symmetry of matrix-valued differential operators which we expect to be useful in other cases as well. We have relegated all the proofs that only involve Gegenbauer polynomials to the appendices.

\begin{rmk}
All proofs are direct and only use special functions. Anyway we have benefited of the use of computer algebra 
as Maple, in order to come up with explicit conjectures. This means that we have checked many of our results up to a sufficiently large size of the matrices involved. For the reader's convenience the worksheet is available via the first author's webpage.\footnote{\texttt{http://www.math.ru.nl/\~{}koelink}}
\end{rmk}

\subsection*{Acknowledgement.} We thank Ruiming Zhang for asking the questions at 
a presentation by one of us (EK) on \cite{KoelvPR1}, \cite{KoelvPR2}, which eventually led to this paper. 
The research of AMdlR is partially supported by  MTM2012-36732-C03-03 (Ministerio de Econom\'ia y Competitividad),
FQM-262, FQM-4643, FQM-7276 (Junta de Andaluc\'ia) and Feder Funds (European Union).
Part of the work was done while EK was visiting Universidad Nacional C\'ordoba and while AMdlR or PR was visiting Radboud Universiteit.
We thank both universities for their hospitality.
PR's visit to Nijmegen was supported by a NWO-Visiting Grant 040.11.366. 
PR was also supported by the Coimbra Group Scholarships Programme at KULeuven in the period February-May 2014.

\section{The weight function and symmetric differential operators}\label{sec:weightfunctionandsymmetricdiffop}

The weight function is introduced by defining its matrix entries, which are taken with respect to the standard orthonormal 
basis $\{e_0,e_1,\cdots, e_{2\ell}\}$ of $\C^{2\ell+1}$ for $\ell\in\frac12 \N$. 
We suppress $\ell$ from the notation as much as possible, which occurred initially as the spin of a representation of $\SU(2)$, see \cite{KoelvPR1}. 
After defining the weight we discuss its LDU-decomposition, and we study commuting symmetric matrix-differential operators 
for the weight. These matrix-differential operators follow directly using the fact that we want
the transmutation property with the derivative. Then we 
obtain in this algebra a symmetric second-order matrix-valued differential operator on which we can 
apply the Darboux factorisation, which gives the entries for the matrix-valued Pearson 
equation as in \cite{CantMV}. 

\begin{defn}\label{def:weightW}
For $\ell\in\frac12\N$ and $\nu> 0$,  we define the $(2\ell+1)\times (2\ell+1)$-matrix-valued functions 
$W^{(\nu)}$ by 
\begin{equation}\label{eq:defWlambda}
\begin{split}
\left(W^{(\nu)}(x)\right)_{m,n}= &\, (1-x^2)^{\nu-1/2} \hskip-.3truecm \sum_{t=\max(0, n+m-2\ell)}^{m} 
\hskip-.3truecm  \al_{t}^{(\nu)}(m,n) 
C_{m+n-2t}^{(\nu)}(x)\\ 
\alpha_{t}^{(\nu)}(m,n) =& \, (-1)^{m} 
\frac{n!\, m!\, (m+n-2 t)!}{t!\, (2\nu)_{m+n-2t}\, (\nu)_{n+m-t}}
\frac{(\nu)_{n-t} (\nu)_{m-t}}{(n-t)!\, (m-t)!}\frac{(n+m-2t+\nu)}{(n+m-t+\nu)}\\ 
& \, \times (2\ell-m)!(n-2\ell)_{m-t}(-2\ell-\nu)_{t} \frac{(2\ell+\nu)}{(2\ell)!} 
\end{split}
\end{equation}
where $n,m\in\{0,1,\cdots, 2\ell\}$ and $n\geq m$. The matrix is extended to a symmetric matrix, 
$\left(W^{(\nu)}(x)\right)_{m,n}=\left(W^{(\nu)}(x)\right)_{n,m}$.
Finally, put $W^{(\nu)}(x) = (1-x^2)^{\nu-1/2} W^{(\nu)}_{\pol}(x)$.
\end{defn}

In order to show that matrix-valued weight  fits into the general theory of 
matrix-valued orthogonal polynomials we calculate its LDU-decomposition. 
This is also very useful in establishing symmetry of matrix-valued 
differential operators and in the study of the derivative in this context. 

\begin{thm}\label{thm:LDUweight}
For $\nu> 0$, $W^{(\nu)}(x)$ has the following LDU-decomposition
\begin{equation*}
W^{(\nu)}(x)=
% (1-x^2)^{\nu-1/2}
L^{(\nu)}(x)T^{(\nu)}(x)L^{(\nu)}(x)^{t}, \qquad x\in(-1,1),
\end{equation*} 
where $L^{(\nu)}\colon [-1,1]\to M_{2\ell+1}(\C)$ is the unipotent lower triangular matrix-valued polynomial 
\[
\bigl(L^{(\nu)}(x)\bigr)_{m,k}=\begin{cases} 0 & \text{if } m<k \\
% \displaystyle{\frac{m!}{(m-k)!\, k!}  \rFs{2}{1}{-m+k,m+k+2\nu}{\frac12+k+\nu}{\frac{1-x}{2}}} & \text{if } m\geq k.
\displaystyle{\frac{m!}{k! (2\nu+2k)_{m-k}} C^{(\nu+k)}_{m-k}(x)} & \text{if } m\geq k.
\end{cases}
\]
and $T^{(\nu)}\colon (-1,1)\to M_{2\ell+1}(\C)$ is the diagonal matrix-valued function 
\begin{equation*}
\begin{split}
\bigl(T^{(\nu)}(x)\bigr)_{k,k}\, =\, t^{(\nu)}_{k}\,  (1-x^2)^{k+\nu-1/2}, \quad 
 t^{(\nu)}_{k}\, =\, \frac{k!\, (\nu)_k}{(\nu+1/2)_k}
\frac{(2\nu + 2\ell )_{k}\, (2\ell+\nu)}{(2\ell - k+1)_k\, ( 2\nu + k - 1)_{k}}.
\end{split}
\end{equation*}
\end{thm}

Theorem \ref{thm:LDUweight} is proved in Appendix 
\ref{app:proofthmLDUdecompW}, and it is a straightforward extension of the 
proof in \cite{KoelvPR2}. 
$L(x)$ is invertible, its inverse being an unipotent lower triangular matrix as well.
Cagliero and Koornwinder \cite{CaglK} give the matrix entries of $L(x)^{-1}$ explicitly in 
terms of Gegenbauer polynomials with negative $\nu$, see \eqref{eq:inverseLbyCK}. 
Note that Theorem \ref{thm:LDUweight} gives that $W^{(\nu)}(x)>0$ for $x\in (-1,1)$.
It is possible to extend Definition \ref{def:weightW} and Theorem \ref{thm:LDUweight} 
to $-\frac12< \nu\leq 0$, but for $\nu=0$ the weight $W^{(\nu)}$ is indefinite and 
for $-\frac12< \nu < 0$ it is has non-trivial signature depending on the size. 
So in this paper we assume $\nu>0$. 

For matrix-valued functions $P$ and $Q$ with, say, $C([-1,1])$-entries we define for $\nu>0$ 
the matrix-valued inner product 
\begin{equation}\label{eq:defMVinnerproductnu}
\langle P, Q\rangle^{(\nu)} = \int_{-1}^1 P(x) \, W^{(\nu)}(x) \, (Q(x))^\ast \, dx \in M_{2\ell+1}(\C). 
\end{equation}
Note that the integral of the entries are well-defined.  

A matrix-valued differential operator acts from the right; 
\begin{equation}\label{eq:defMVdiffoperator}
\bigl( PD\bigr)(x) = \frac{d^2P}{dx^2}(x) F_2(x) + \frac{dP}{dx}(x) F_1(x) + P(x) F_0(x)
\end{equation}
where $D= \frac{d^2}{dx^2} F_2 + \frac{d}{dx}(x) F_1 + F_0$, $F_i$ are matrix-valued
functions and $P$ is a matrix-valued function with $C^2$-entries. 
The derivatives in \eqref{eq:defMVdiffoperator} of a matrix-valued function are
taken entry-wise. A matrix-valued differential operator $D$ is symmetric for the weight $W^{(\nu)}$
if 
$\langle PD, Q\rangle^{(\nu)} = \langle P, QD\rangle^{(\nu)}$
for all matrix-valued functions $P$ and $Q$ with $C^2([-1,1])$-entries. 

\begin{thm}\label{thm:thmdifferentialoperators}
For $\nu >0$, let $D^{(\nu)}$ and $E^{(\nu)}$ be the matrix-valued differential operators
\begin{gather*}
D^{(\nu)}\,=\,\left(\frac{d^2}{dx^2}\right) (1-x^2)+\left(\frac{d}{dx}\right)(C-x(2\ell+2\nu+1)I)-V,\\
E^{(\nu)}\,=\,\left(\frac{d}{dx}\right)(xB_1+B_0)+A_0^{(\nu)},
\end{gather*}
where the matrices $C$, $V$, $B_0^{(\nu)}$, $B_1^{(\nu)}$ and $A_0^{(\nu)}$ are given by
\begin{gather*}
C=\sum_{i=0}^{2\ell-1} (2\ell-i) E_{i,i+1} + \sum_{i=1}^{2\ell} i E_{i,i-1},
\qquad 
V= -\sum_{i=0}^{2\ell} i(2\ell-i) E_{i,i} \\
2\ell\, B_0=\sum_{i=0}^{2\ell-1} (2\ell-i) \, E_{i,i+1} - \sum_{i=1}^{2\ell} i\,E_{i,i-1},\qquad 
\ell \, B_1=-\sum_{i=0}^{2\ell} (\ell-i) \, E_{i,i}, \\ 
\ell \, A^{(\nu)}_0=\sum_{i=0}^{2\ell} ((\ell+1)(i-2\ell)  - (\nu-1)(\ell-i)) \, E_{i,i}.
\end{gather*}
Then $D^{(\nu)}$ and $E^{(\nu)}$ are symmetric with respect to the weight $W^{(\nu)}$, and $D^{(\nu)}$ and $E^{(\nu)}$ commute. 
\end{thm}

The operator $D^{(\nu)}$ matches the case $\nu=1$ in \cite[Thm.~3.1]{KoelvPR2}. The explicit expression of $E^{(\nu)}$ for $\nu=1$ corrects a mistake in \cite[Thm.~3.1]{KoelvPR2} and upon a change of variables $x=1-2u$ matches \cite[Cor.~4.1]{KoelvPR2}. The explicit expressions for $D^{(\nu)}$ and $E^{(\nu)}$ for $\nu=1$ are given 
in \cite{KoelvPR1}, \cite{KoelvPR2}. With this definition we have   
$E^{(\nu+1)}\circ \frac{d}{dx}= \frac{d}{dx} \circ E^{(\nu)}$ and 
$\bigl( D^{(\nu+1)}-(2\ell+2\nu+1)\Id\bigr)\circ \frac{d}{dx}= \frac{d}{dx} \circ D^{(\nu)}$.
In order to prove the statement on the symmetry 
in Theorem \ref{thm:thmdifferentialoperators} we use Theorem \ref{thm:LDUweight}. 

We next look for a second-order matrix-valued differential operator  generated by the 
commuting operators $D^{(\nu)}$ and $E^{(\nu)}$ having no constant term, i.e. $F_0=0$ in the notation of \eqref{eq:defMVdiffoperator}. The reason is that in the matrix-valued case 
the constant term $F_0$ cannot be moved into the eigenvalue matrix, see e.g. 
Theorem \ref{thm:thmdifferentialoperatorsP_d}, unless $F_0$ is a multiple of the identity. 
A straightforward calculation gives 
\begin{equation}\label{eq:relDPhiPsi-DnuEnu}
\begin{split}
&\mathscr{D}^{(\nu)} =(E^{(\nu)})^{2}+(2\ell+2)E^{(\nu)}+\left(\frac { \left( \ell+\nu \right)^{2}}{{\ell}^{2}}\right)D^{(\nu)}+
\frac {\nu \left( \nu-1 \right)  \left( 2\,\ell+\nu+1 \right)  \left(\nu+2\,\ell \right) }{{\ell}^{2}} \Id \\
&\qquad\qquad  
(P\mathscr{D}^{(\nu)} ) =  \frac{d^2P}{dx^2}(x) \Bigl( \Phi^{(\nu)}(x) \Bigr)^\ast + 
\frac{dP}{dx}(x) \Bigl( \Psi^{(\nu)}(x) \Bigr)^\ast 
\end{split}
\end{equation}
and this defines matrix-valued polynomials $\Phi^{(\nu)}$ and $\Psi^{(\nu)}$ of degree $2$ and $1$. The explicit expressions for $\Phi^{(\nu)}$ and $\Psi^{(\nu)}$ are given in 
\eqref{eq:PhuPsiexplicit} and \eqref{eq:PhuPsiexplicit2}. 
In particular we find that the $\nu$-dependence of 
$\Phi^{(\nu)}(x)= \Phi(x) + \frac{(\ell+\nu)^2}{\ell^2}(1-x^2)\Id$ is rather simple. 
Now $\mathscr{D}^{(\nu)}$ is factored as derivation followed by a first order matrix-valued differential operator,
and in order to study this operator we need the analogue of the Pearson 
equation \eqref{eq:Pearsonscalar}.

\begin{thm}\label{thm:Pearson}
We have 
\begin{gather*}
\frac{d(W^{(\nu)}\Phi^{(\nu)})}{dx}(x) =W^{(\nu)}(x)\Psi^{(\nu)}(x), \quad 
W^{(\nu+1)}(x)=c^{(\nu)}W^{(\nu)}(x)\Phi^{(\nu)}(x),   \\
c^{(\nu)} = \frac{(2\nu+1)(2\ell+\nu+1) \ell^{2}}{\nu (2\nu+2\ell+1) (2\ell+\nu) (\ell+\nu)}. 
\end{gather*}
\end{thm}

The matrix-valued Pearson equation of Theorem \ref{thm:Pearson} is much more involved than
its scalar companion \eqref{eq:Pearsonscalar}, and it fits into the framework 
of \cite[\S 3]{CantMV}. 

The space of matrix-valued functions with continuous entries with respect to \eqref{eq:defMVinnerproductnu} forms a pre-Hilbert C$^\ast$-module, see 
\cite{Lanc}, with $M_{2\ell+1}(\C)$ the corresponding (finite-dimensional) C$^\ast$-algebra. Note that we consider the pre-Hilbert C$^\ast$-module 
as a left module for $M_{2\ell+1}(\C)$ and the inner product to be conjugate linear in the second variable, in contrast with \cite{Lanc}. 
Let $\cH^{(\nu)}$ be the Hilbert C$^\ast$-module which is the completion \cite[p.~4]{Lanc}, then 
$\frac{d}{dx}\colon \cH^{(\nu)} \to \cH^{(\nu+1)}$ is an unbounded operator with dense domain and dense range in  
$\cH^{(\nu)}$ and $\cH^{(\nu+1)}$. In general, a linear operator from
one Hilbert C$^\ast$-module to another does not necessarily have an adjoint, but in this case
it does.  

\begin{cor}\label{cor:adjointddx}
(i) Define the first order matrix-valued differential operator $S^{(\nu)}$ by
\begin{equation*}
\bigl( QS^{(\nu)}\bigr)(x)\, =\, \frac{dQ}{dx}(x) \bigl(\Phi^{(\nu)}(x)\bigr)^\ast + Q(x)\bigl(\Psi^{(\nu)}(x)\bigr)^\ast,
\end{equation*}
then 
$\langle \frac{dP}{dx}, Q\rangle^{(\nu+1)} = - c^{(\nu)} \langle P, QS^{(\nu)}\rangle^{(\nu)}$
for matrix-valued functions $P$ and $Q$ with $C^1([-1,1])$-entries, with $c^{(\nu)}$ as in 
Theorem \ref{thm:Pearson}. \\
(ii) $\langle P\mathscr{D}^{(\nu)}, Q\rangle^{(\nu)}= -c^{(\nu)} 
\langle \frac{dP}{dx}, \frac{dQ}{dx}\rangle^{(\nu+1)}$
for matrix-valued functions $P$ and $Q$ with $C^2([-1,1])$-entries. 
\end{cor}

Note that Corollary \ref{cor:adjointddx}(ii) gives the symmetry of $\mathscr{D}^{(\nu)}$, 
which is a consequence of Theorem  \ref{thm:Pearson}. However, we use the 
symmetry of $\mathscr{D}^{(\nu)}$ in the proof of Theorem \ref{thm:Pearson}, so that 
we need another proof of the symmetry of $\mathscr{D}^{(\nu)}$. The required
symmetry follows from Theorem \ref{thm:thmdifferentialoperators} and 
\eqref{eq:relDPhiPsi-DnuEnu}. 

Having factorised $\mathscr{D}^{(\nu)}$ as a product of a lowering operator and a rising 
operator in \eqref{eq:relDPhiPsi-DnuEnu} and Corollary \ref{cor:adjointddx}, we can take its
Darboux transform. The Darboux transform does not give the $\mathscr{D}^{(\nu+1)}$
up to an affine transformation but it is an element the algebra generated by $E^{(\nu+1)}, D^{(\nu+1)}$, namely
\begin{align*}
&\frac{d}{dx}\circ S^{(\nu)}=\frac{d^{2}}{dx^{2}} \Phi^{(\nu)}(x)^{*}+\frac{d}{dx} 
\left( \frac{d}{dx}\Phi^{(\nu)}(x)^*+\Psi^{(\nu)}(x)^{*}\right)+  \left(\frac{d}{dx}\Psi^{(\nu)}(x)^{*}\right)\\
&=\left(E^{(\nu+1)}\right)^{2}+(2\ell+2)E^{(\nu+1)}+\left(\frac{\ell+\nu}{\ell}\right)^{2} D^{(\nu+1)}+ 
\left(\frac{\nu (\nu-1) (2\ell+\nu) (2\ell +\nu+1)}{\ell^2}\right) \text{Id}.
\end{align*}
Darboux transformations for matrix-valued differential operators require more study, see e.g. \cite{AlvaAGAMM}, \cite{Casp}, \cite{CastG}, \cite{Grun}.

\begin{prop}\label{prop:commutantW} 
The commutant algebra $A^{(\nu)}=\{ T\in M_{2\ell+1}(\C)\mid  [T,W^{(\nu)}(x)]=0 \, \forall x\in(-1,1)\}$ 
is generated by $J$, where $J\in M_{2\ell+1}(\C)$ is the 
self-adjoint involution defined by $J\colon e_j \mapsto e_{2\ell-j}$. 
\end{prop}

The commutant algebra $A^{(\nu)}$ is related to orthogonal decompositions, and 
Proposition \ref{prop:commutantW} states that there is an orthogonal decomposition 
with respect to the $\pm 1$-eigenspaces of $J$. 
General non-orthogonal decompositions are governed by the real vector space 
$\mathscr{A}^{(\nu)} = \{ T\in M_{2\ell+1}(\C)\mid  TW^{(\nu)}(x)=W^{(\nu)}(x)T^\ast \, \forall x\in(-1,1)\}$, see \cite{GrunT}, \cite{TiraZ}. In \cite{KoelR} we show that $\mathscr{A}^{(\nu)}$ equals the Hermitian elements of 
$A^{(\nu)}$, so that there is no further non-orthogonal decomposition. Proposition \ref{prop:commutantW} is proven in Section \ref{ssec:commutant}.

%%%%%%%%%%%%%%%%%%%%%%%%%%%%%%%%%%%%%%%%%%%%%%%%%%%%%%%%%%%%%%%%%%%
%%%%%SECTION%%%%%%%%%%%%%%%%%%%%%%%%%%%%%%%%%%%%%%%%%%%%%%%%%%%%%%%
%%%%%%%%%%%%%%%%%%%%%%%%%%%%%%%%%%%%%%%%%%%%%%%%%%%%%%%%%%%%%%%%%%%
\section{The matrix-valued Gegenbauer polynomials and their properties}\label{sec:MVGegenbauerpolsandprops}

Since the matrix weight function $W^{(\nu)}$ is strictly positive definite on $(-1,1)$ one
can associate matrix-valued orthogonal polynomials, see e.g. \cite{DamaPS}, \cite{GrunT}, \cite{HeckvP}. 
Denote the corresponding monic orthogonal polynomials with respect to 
the matrix-valued weight function $W^{(\nu)}$ on $(-1,1)$
by $P_n^{(\nu)}$, i.e.
\begin{equation}\label{eq:defmonicMVorthopol}
\begin{split}
&\qquad \qquad \int_{-1}^1 P_n^{(\nu)}(x) \, W^{(\nu)}(x) \bigl( P_m^{(\nu)}(x)\bigr)^\ast \, dx \, = \, \de_{n,m} H_n^{(\nu)}, \\
& P_n^{(\nu)}(x) = x^n \, \Id + x^{n-1} P^{(\nu)}_{n,n-1} + \cdots +  x P^{(\nu)}_{n,1} + P^{(\nu)}_{n,0}, 
\qquad P^{(\nu)}_{n,n} = \Id,\  P^{(\nu)}_{n,i} \in M_{2\ell+1}(\C), 
\end{split}
\end{equation}
where $H_n^{(\nu)}$ is a strictly positive definite matrix. Note that $H_0^{(\nu)}$ can be calculated 
using the explicit expression of Definition \ref{def:weightW} and the orthogonality relations 
\eqref{eq:orthoscalarGegenbauer}, which gives the special case $n=0$ of Theorem \ref{thm:squarenormRodrigues}(i). 
In case $Q_n$ is another set of matrix-valued orthogonalpolynomials with respect to 
weight function $W^{(\nu)}$ on $(-1,1)$, then there exist invertible matrices $E_n$ so that $Q_n(x) = E_nP_n(x)$
for all $x$ and all $n$, see \cite{DamaPS}, \cite{GrunT}. 

\begin{thm}\label{thm:squarenormRodrigues} 
\text{\rm (i)} The squared norm $H^{(\nu)}_n$ in \eqref{eq:defmonicMVorthopol} is given by the diagonal matrix 
\begin{gather*}
\bigl( H^{(\nu)}_n\bigr)_{k,k} = \sqrt{\pi}\,  \frac{\Ga(\nu+\frac12)}{\Ga(\nu+1)}
\frac{\nu(2\ell+\nu+n)}{\nu+n} \frac{n!\, (\ell+\frac12+\nu)_n (2\ell+\nu)_n(\ell+\nu)_n}{(2\ell+\nu+1)_n(\nu+k)_n(2\ell+2\nu+n)_n(2\ell+\nu-k)_n} \\
\times \frac{k!\, (2\ell-k)!\, (n+\nu+1)_{2\ell}}{(2\ell)!\, (n+\nu+1)_k (n+\nu+1)_{2\ell-k}}
\end{gather*}
\text{\rm (ii)} $\displaystyle{\frac{dP^{(\nu)}_n}{dx}(x) = n\, P^{(\nu+1)}_{n-1}(x)}$. \par\noindent
\text{\rm (iii)} The following Rodrigues formula holds:
\begin{gather*}
P^{(\nu)}_n(x) =  G_n^{(\nu)}\, \left(\frac{d^nW^{(\nu+n)}}{dx^n}(x)\right)\, 
W^{(\nu)}(x)^{-1} \\
\bigl(G_n^{(\nu)}\bigr)_{j,k} = \de_{j,k} \frac{(-1)^n (\nu)_n (\ell+\nu+\frac12)_n (\ell+\nu)_n (2\ell+\nu)_n}
{(\nu+\frac12)_n(\nu+k)_n(2\ell+\nu+1)_n(2\ell+2\nu+n)_n(2\ell+\nu-k)_n}.
\end{gather*}
\end{thm}

Note that the Rodrigues formula has a compact nature very similar
to the scalar case \cite[(9.8.27)]{KoekLS}, \cite[(1.8.23)]{KoekS}, \cite[(4.5.12)]{Isma} and works for any size. 
It differs from the Rodrigues formula
for the irreducible $2\times 2$-cases and $\nu=1$ in  \cite[\S 8]{KoelvPR1}.  

\begin{thm}\label{thm:thmdifferentialoperatorsP_d}
For every integer $n\geq 0$,
the monic matrix-valued Gegenbauer polynomials are eigenfunctions
of $D^{(\nu)}$ and $E^{(\nu)}$; 
\begin{gather*}
P^{(\nu)}_nD^{(\nu)}=\La_n(D^{(\nu)})P^{(\nu)}_n,\quad \La_n(D^{(\nu)})=
-n(2\ell+2\nu+n)I-V,
\\
P^{(\nu)}_nE^{(\nu)}=\La_n(E^{(\nu)})P^{(\nu)}_n, \quad  
\La_n(E^{(\nu)})= A_0^{(\nu)} + n B_1.
\end{gather*}
\end{thm}

It follows from \eqref{eq:relDPhiPsi-DnuEnu} that $P^{(\nu)}_n$ are eigenfunctions
to $\mathscr{D}^{(\nu)}$, where the eigenvalue matrix is obtained using the same
combination.

Matrix-valued orthogonal polynomials satisfy a three-term recurrence relation, see 
e.g. \cite[Lemma 2.6]{DamaPS}, \cite[(2.1)]{GrunT}. 

\begin{thm}\label{thm:three_term_for_Pn}
The monic matrix-valued orthogonal polynomials
$P^{(\nu)}_n$ satisfy the three-term recurrence relation
\begin{equation}
xP^{(\nu)}_{n}(x)=P^{(\nu)}_{n+1}(x)+B_n^{(\nu)}P^{(\nu)}_{n}(x)+
C^{(\nu)}_{n}P^{(\nu)}_{n-1}(x) 
\end{equation}
where the matrices $B^{(\nu)}_n$, $C^{(\nu)}_n$ are given by
\begin{align*}
B^{(\nu)}_n&=\sum_{j=1}^{2\ell}  \frac{j(j+\nu-1)}{2(j+n+\nu-1)(j+n+\nu)}E_{j,j-1} + 
\\
&\qquad\qquad  
 \sum_{j=0}^{2\ell-1}
 \frac{(2\ell-j)(2\ell-j+\nu-1)}{2(2\ell-j+n+\nu-1)(2\ell+n-j+\nu)}E_{j,j+1}  \\
C^{(\nu)}_n&=\sum_{j=0}^{2\ell} \frac{n(n+\nu-1)(2\ell+n+\nu)(2\ell+n+2\nu-1)}
{4(2\ell+n+\nu-j-1)(2\ell+n+\nu-j)(j+n+\nu-1)(j+n+\nu)} E_{j,j}.
\end{align*}
\end{thm}

Note that $JB^{(\nu)}_n= B^{(\nu)}_nJ$ and $JC^{(\nu)}_n=C^{(\nu)}_nJ$, which essentially
follows from Proposition \ref{prop:commutantW}, see \cite[Lem.~3.1]{KoelR}.
The coefficient matric $C_n^{(\nu)}$ in Theorem \ref{thm:three_term_for_Pn} 
follows from Theorem \ref{thm:squarenormRodrigues}, but for the coefficient
$B^{(\nu)}_n$ we need the one-but-leading coefficient of $P_n^{(\nu)}$ for which
we use the shift operators. In \cite{KoelvPR2} we calculated the one-but-leading coefficients
using Tirao's matrix-valued hypergeometric functions \cite{TiraPNAS}, and we note that 
a similar expression for the rows of $P^{(\nu)}_n$ in terms of matrix-valued hypergeometric functions as in \cite[\S 4]{KoelvPR2} is possible. 

The matrix-entries of the monic matrix-valued orthogonal polynomials 
can be expressed in terms of scalar orthogonal polynomials. In this expression Racah polynomials
\cite[\S 4]{Wils}, \cite[\S 9.2]{KoekLS}  and 
Gegenbauer polynomials occur. The Gegenbauer polynomials with negative parameter $\nu$ 
arise from $L(x)^{-1}$ and have to be interpreted as in \cite{CaglK}. 

\begin{thm}\label{thm:calPnasRacahGegenbauer} 
The monic matrix-valued Gegenbauer polynomials $P^{(\nu)}_n(x)$ have the 
explicit expansion
\begin{align*}
(P^{(\nu)}_n(x))_{k,i}&= \frac{(-2)^n}{i!} \, c^{(\nu)}_{k,0}(n) 
\sum_{j=i}^{\min(2\ell,n+k)} \frac{ (n+k-j)!(-2\ell)_j(-1)^j(-n-k)_j }{(2\nu+2j)_{n+k-j} (2\nu+j+i-1)_{j-i} (2\nu+2\ell)_j }
\, \\
& \quad \times R_k(\lambda(j);-2\ell-1,-n-k-\nu,\nu-1,\nu-1)\, C^{(\nu+j)}_{n+k-j}(x) \, C^{(1-\nu-j)}_{j-i}(x), 
\\ 
c^{(\nu)}_{k,0}(n)  &=  
\frac{2^{-n}(\nu)_n (2\ell+2\nu)_n (n+k)!}{(\nu+k)_n (2\ell+\nu-k)_n(2\nu)_{n+k}}.
\end{align*}
\end{thm}

Theorem \ref{thm:calPnasRacahGegenbauer} is proved in Section \ref{ssec:eigenfunctionsMVOP}. 

The right hand side in Theorem \ref{thm:calPnasRacahGegenbauer} 
is not obviously a polynomial of degree at most $n$ in case
$k>i$. In particular, the coefficients of $x^p$ with $p>n$ are zero. 
In particular, for $k>i$ the leading coefficient of the right hand side is zero, and this gives
\begin{equation}
\begin{split}
&\sum_{j=i}^{2\ell} \frac{ (\nu+j)_{n+k-j} (-2\ell)_j(-1)^j(-n-k)_j }{(2\nu+2j)_{n+k-j} (2\nu+j+i-1)_{j-i} (2\nu+2\ell)_j }
\\ &
\times \frac{(1-\nu-j)_{j-i}}{(j-i)!}
R_k(\lambda(j);-2\ell-1,-n-k-\nu,\nu-1,\nu-1) = 0.
\end{split}
\end{equation}

% \begin{rmk}
% Another explicit expression for the matrix-valued Gegenbauer polynomials can be obtained by expressing the rows in terms of matrix-valued hypergeometric functions \cite{TiraPNAS} analogously as for the case $\nu=1$ \cite{KoelvPR2}. 
% \end{rmk}
\begin{rmk}
From the orthogonality relations and Theorem \ref{thm:Pearson} or Theorem \ref{thm:squarenormRodrigues} it follows that the polynomial $P_n^{(\nu+k)}$ can be expanded in polynomials $P_m^{(\nu)}$ with $n-2k\leq m\leq n$. However, in general, the coefficients do not seem to have simple expressions. The case $k=1$ can be done easily by differentiating the three-term recurrence relation of Theorem \ref{thm:three_term_for_Pn} and using Theorem \ref{thm:squarenormRodrigues}(ii).
\end{rmk}

\begin{rmk}
The weight matrices $W^{(\nu+1)}$ is obtained from $W^{(\nu)}$ by a Christoffel transformation, given by multiplication by the polynomial $\Phi^{(\nu)}$ of degree two, see Theorem \ref{thm:Pearson}. Therefore our Gegenbauer polynomials give a nontrivial example of arbitrary size for the theory developed in \cite{AlvaAGAMM}, see also \cite[Example 3]{AlvaAGAMM}.
\end{rmk}

\begin{rmk}
Pacharoni and Zurri\'an \cite{PachZ} introduce $2\times 2$ matrix-valued Gegenbauer polynomials. Using Proposition \ref{prop:commutantW} we have irreducible $2\times 2$ Gegenbauer polynomials upon restricting to the $\pm 1$-eigenspaces of $J$ for the cases $\ell=1,3/2,2$. The cases $\ell=1$ and $\ell=2$ can be connected to the results in  \cite{PachZ}, but the 
case $\ell=3/2$ not. 
\end{rmk}

%%%%%%%%%%%%%%%%%%%%%%%%%%%%%%%%%%%%%%%%%%%%%%%%%%%%%%%%%%%%%%%%%%%
%%%%%SECTION%%%%%%%%%%%%%%%%%%%%%%%%%%%%%%%%%%%%%%%%%%%%%%%%%%%%%%%
%%%%%%%%%%%%%%%%%%%%%%%%%%%%%%%%%%%%%%%%%%%%%%%%%%%%%%%%%%%%%%%%%%%
\section{Differential operators}\label{sec:differentialoperators}

In this section we prove the statements of Section \ref{sec:weightfunctionandsymmetricdiffop}
except for some technical statements, which are deferred to the appendix, and the commutativity statement of Theorem \ref{thm:thmdifferentialoperators} 
and Proposition \ref{prop:commutantW}.
These statements are easier to derive using the matrix-valued orthogonal polynomials. 
Section \ref{ssec:differentialoperatorsconjugation} is of a general nature, and is 
next applied to this particular situation in Section \ref{ssec:explicitdifferentialoperatorsconjugation}. 
The proof of the Pearson type result is the most technical, see Section 
\ref{ssec:Pearsoneq}. 

%%%%%SECTION%%%%%%%%%%%%%%%%%%%%%%%%%%%%%%%%%%%%%%%%%%%%%%%%%%%%%%%
%%%%%%%%%%%%%%%%%%%%%%%%%%%%%%%%%%%%%%%%%%%%%%%%%%%%%%%%%%%%%%%%%%%
\subsection{Differential operators and conjugation}\label{ssec:differentialoperatorsconjugation}

We consider matrix-valued differential operators of the form
$\frac{d^2}{dx} F_2(x) + \frac{d}{dx} F_1(x) + F_0(x)$, which act from the right on
matrix-valued functions $G$ which have $C^2$-matrix entries by
\begin{equation}\label{eq:defMVDiffOp}
\bigl( GD\bigr)(x) = \frac{d^2G}{dx^2}(x) F_2(x) + \frac{dG}{dx}(x) F_1(x) + G(x) F_0(x).
\end{equation}
In \eqref{eq:defMVDiffOp} the derivatives are taken entry-wise. Moreover, we assume that 
all entries of $F_i$, $i=0,1,2$, are $C^2([a,b])$. In the application for this paper, $F_i$'s are  polynomials. 

Next we assume that we have matrix-valued weight function $W$ on $(a,b)$, so that 
its entries are $C^2((a,b))$, and we allow for an integrable singularity at the end points. 
The matrix-valued operator $D$ is symmetric with respect to  $W$ if
for all matrix-valued $C^2([a,b])$-functions $G$, $H$ we have 
\begin{equation}\label{eq:symmMVDiff}
\int_a^b \bigl( GD\bigr)(x) W(x) H^\ast(x)\, dx = 
\int_a^b G(x)  W(x) \left( \bigl( HD\bigr)(x)\right)^\ast \, dx
\end{equation}
as long as both integrals exist. For the explicit weight functions $W^{(\nu)}$ this
is valid. 

Lemma \ref{lem:symmetricD} is due to Dur\'an and Gr\"unbaum  \cite[Thm~3.1]{DuraG}, 
and can be proved directly by integration by parts. 

\begin{lem}\label{lem:symmetricD}
$D$ as in \eqref{eq:defMVDiffOp} is symmetric with respect to $W$ if and only if
the boundary conditions 
\begin{gather*}
\lim_{x\to a} F_2(x)W(x) = 0 = \lim_{x\to b} F_2(x)W(x), \\ 
\lim_{x\to a} F_1(x)W(x) - \frac{d(F_2W)}{dx}(x) = 0 = 
\lim_{x\to b} F_1(x)W(x) - \frac{d(F_2W)}{dx}(x)
\end{gather*}
and the symmetry conditions 
\begin{gather*}
F_2(x)W(x) = W(x) \bigl( F_2(x)\bigr)^\ast, \qquad 
2 \frac{d(F_2W)}{dx}(x) - F_1(x)W(x) = W(x) \bigl( F_1(x)\bigr)^\ast, \\
\frac{d^2(F_2W)}{dx^2}(x) - \frac{d(F_1W)}{dx}(x) + F_0(x) W(x) = W(x) \bigl( F_0(x)\bigr)^\ast
\end{gather*}
for all $x\in (a,b)$  hold. 
\end{lem}

Note that such symmetric matrix-valued differential operators form a vector space
and that the product of two symmetric first order matrix-valued differential
operators is a symmetric second order matrix-valued differential
operator. 

We now consider the differential operator under conjugation.
We assume $L$ is a $C^2([a,b])$ matrix-valued function such that $L(x)$ is a unipotent 
lower-triangular matrix. In particular, it follows that its inverse $L(x)^{-1}$ is 
a unipotent lower-triangular matrix with $C^2([a,b])$-entries. 
Then $L^\ast(x)$ and $(L^\ast(x))^{-1}$ are unipotent upper-triangular matrix with $C^2([a,b])$-entries.
Note that $L$ as in Theorem \ref{thm:LDUweight} satisfies these conditions. 
Let $\tilde{D} = \frac{d^2}{dx^2} \tilde{F}_2 + \frac{d}{dx} \tilde{F}_1 + \tilde{F}_0$
be the second-order matrix-valued differential operator obtained by conjugation of $D$ by $L$, i.e.
\begin{equation}\label{eq:conjugatedD}
\frac{d^2(GL)}{dx^2}(x) \tilde{F}_2(x) + \frac{d(GL)}{dx}(x) \tilde{F}_1(x) + 
(GL)(x)\tilde{F}_0(x) = 
\bigl( GD\bigr)(x) L(x)
\end{equation}
for all $C^2$-matrix-valued functions $G$. Comparing \eqref{eq:defMVDiffOp}
and \eqref{eq:conjugatedD} 
we obtain 
(as matrix-valued functions)
\begin{equation}\label{eq:relationsFtilde}
F_2L=L\tilde{F}_2, \qquad
F_1L= 2\frac{dL}{dx} \tilde{F}_2 + L\tilde{F}_1, \qquad
F_0L = \frac{d^2L}{dx^2} \tilde{F}_2 + \frac{dL}{dx} \tilde{F}_1 + L\tilde{F}_0
\end{equation}
By symmetry, interchanging $F_i\leftrightarrow\tilde{F}_i$ and $L\leftrightarrow L^{-1}$, 
we get the expressions for $F_i$ in terms of $\tilde{F}_i$. 
This can also be obtained by solving for $\tilde{F}_i$ from \eqref{eq:relationsFtilde}
and using the expressions for the first and second order derivatives of $L^{-1}$ 
obtained from deriving $LL^{-1}=\text{I}$.

\begin{prop}\label{prop:symetryunderconjugation} 
With the assumptions in Section \ref{ssec:differentialoperatorsconjugation} 
and assuming additionally for $x\in(a,b)$ that $W(x)=L(x)V(x)L^\ast(x)$, 
then $D$ is symmetric with respect to $W$ on the interval $(a,b)$ 
if and only if $\tilde{D}$ is symmetric with respect to $V$ on the interval $(a,b)$. 
\end{prop}

Note that in particular the entries of $V$ are again in $C^2((a,b))$. 

\begin{proof}
Note that symmetry of $D$ with respect to $W$ is given by  
\eqref{eq:symmMVDiff}. Plugging in the decomposition of $W$ gives 
\[
\int_a^b ((GD)L)(x) V(x) \bigl( H(x)L(x)\bigr)^\ast\, dx = \int_a^b G(x)L(x) V(x) \bigl( (HDL)(x)\bigr)^\ast\, dx
\]
Replacing $GL$ and $HL$ by $\tilde{G}$ and $\tilde{H}$, and using 
\eqref{eq:conjugatedD} we find 
\[
\int_a^b (\tilde{G}\tilde{D})(x) V(x) \bigl( \tilde{H}(x)\bigr)^\ast\, dx = 
\int_a^b \tilde{G}(x) V(x) \bigl( (\tilde{H}\tilde{D})(x)\bigr)^\ast\, dx
\]
since $L$, being a unipotent matrix with $C^2([a,b])$-entries, has an inverse 
with the same properties. So $\tilde{D}$ is symmetric for $V$ on $(a,b)$.

Since we can interchange the roles of $(D,W)$ and $(\tilde{D},V)$ by flipping
$L\leftrightarrow L^{-1}$ we obtain the result. 
\end{proof}

Proposition \ref{prop:symetryunderconjugation} can also be proved by showing that 
the symmetry relations of Lemma \ref{lem:symmetricD} for $(\tilde{D},V)$ on $(a,b)$
can be obtained from the symmetry relations for $(D,W)$ on $(a,b)$
using \eqref{eq:relationsFtilde}, and vice versa.

%%%%%SECTION%%%%%%%%%%%%%%%%%%%%%%%%%%%%%%%%%%%%%%%%%%%%%%%%%%%%%%%
%%%%%%%%%%%%%%%%%%%%%%%%%%%%%%%%%%%%%%%%%%%%%%%%%%%%%%%%%%%%%%%%%%%
\subsection{Differential operators $E^{(\nu)}$ and $D^{(\nu)}$}\label{ssec:explicitdifferentialoperatorsconjugation}

In order to prove the symmetry of $E^{(\nu)}$ and $D^{(\nu)}$ as
stated in Theorem \ref{thm:thmdifferentialoperators} we
use Lemma \ref{lem:symmetricD} with $L=L^{(\nu)}$ as in 
Theorem \ref{thm:LDUweight}, so that $V$ of Lemma \ref{lem:symmetricD}
identifies with the diagonal weight $T^{(\nu)}$. 
\begin{lem}\label{lem:conjugatedDnuandEnu}
The conjugated differential operators $(D^{(\nu)} - 2\ell E^{(\nu)})\widetilde{\ }$ and $( E^{(\nu)} )\widetilde{\ }$ are given explicitly by
\begin{equation*}
(D^{(\nu)} - 2\ell E^{(\nu)})\widetilde{\ } = \left(\frac{d^2}{dx^2}\right) (1-x^2) + \left(\frac{d}{dx}\right) K_1(x) +K_0, \quad ( E^{(\nu)} )\widetilde{\ } = \frac{d}{dx}  S_1^{(\nu)}(x) + S_0^{(\nu)},
\end{equation*}
where
\begin{align*}
K_1(x)&=-x\sum_{i=0}^{2\ell} 2i\, E_{i,i} + (2\nu+1)I, \qquad K_0 =-\sum_{i=0}^{2\ell}i(2\nu+i)E_{i,i} +2\ell(\nu+2\ell+1)I,\\
S_1^{(\nu)}(x) &= -\frac{1}{2}(1-x^2) 
\sum_{i=1}^{2\ell} \frac{i(i+2\nu-2)(2\nu+i+2\ell-1)}{\ell(2\nu+2i-1)(2\nu+2i-3)} E_{i,i-1} 
- \sum_{i=0}^{2\ell-1} \frac{(i-2\ell)}{2\ell} E_{i,i+1},\\
S_0^{(\nu)}(x)&=x\sum_{i=0}^{2\ell} \frac{i(2\nu+i-2)(2\nu+i+2\ell-1)}{2\ell(2\nu+2i-3)} E_{i,i-1} \nonumber \\
&\qquad\qquad\qquad\qquad  +\sum_{i=0}^{2\ell} \frac{i(2\nu+i-1)-4\ell(\ell+1)-2\ell (\nu-1)}{2\ell} E_{i,i}
\end{align*}
Moreover, $(D^{(\nu)} - 2\ell E^{(\nu)})\widetilde{\ }$ and $( E^{(\nu)} )\widetilde{\ }$ are symmetric with respect to the weight $T^{(\nu)}$.
\end{lem}

%\begin{lem}\label{lem:conjugatedDnuandEnu}
%With the choices as above, the conjugated operators
%\begin{gather*}
%(D^{(\nu)} - 2\ell E^{(\nu)})\widetilde{\ } = ??diagonal?? \\
%( E^{(\nu)} )\widetilde{\ } = \frac{d}{dx}  S_1^{(\nu)}(x) + S_0^{(\nu)}\\
%S_1^{(\nu)}(x) = \frac{-1}{2}(1-x^2) 
%\sum_{i=1}^{2\ell} \frac{i(i+2\nu-2)(2\nu+i+2\ell-1)}{\ell(2\nu+2i-1)(2\nu+2i-3)} E_{i,i-1} 
%+ \sum_{i=0}^{2\ell-1} \frac{(i-2\ell)}{2\ell} E_{i,i+1} \\
%\end{gather*}
%are symmetric with the diagonal weight $T^{(\nu)}$. 
%\end{lem}

Lemma \ref{lem:conjugatedDnuandEnu} is proved in Appendix 
\ref{app:Bproofoflem:conjugatedDnuandEnu}. Lemma \ref{lem:conjugatedDnuandEnu} and Proposition
\ref{prop:symetryunderconjugation}
imply the validity of Theorem \ref{thm:thmdifferentialoperators} except
for the commutativity of the differential operators. 

Using \eqref{eq:relationsFtilde} we 
calculate $(D^{(\nu)} - 2\ell E^{(\nu)})\widetilde{\ }$ and 
$( E^{(\nu)} )\widetilde{\ }$, and since this is an explicit 
calculation involving Gegenbauer polynomials, we do so in 
Appendix \ref{app:Bproofoflem:conjugatedDnuandEnu}.
Then the proof of the symmetry can be reduced using
Lemma \ref{lem:symmetricD}, see Appendix \ref{app:Bproofoflem:conjugatedDnuandEnu}. 
The commutativity is proved in Section \ref{ssec:eigenfunctionsMVOP}.

%%%%%SECTION%%%%%%%%%%%%%%%%%%%%%%%%%%%%%%%%%%%%%%%%%%%%%%%%%%%%%%%
%%%%%%%%%%%%%%%%%%%%%%%%%%%%%%%%%%%%%%%%%%%%%%%%%%%%%%%%%%%%%%%%%%%
\subsection{Pearson equation for the weight $W^{(\nu)}$}\label{ssec:Pearsoneq}

In order to prove the Pearson equations of Theorem 
\ref{thm:Pearson} we use the fact that $\mathscr{D}^{(\nu)}$ is symmetric with respect 
to the weight $W^{(\nu)}$ and we use the LDU-decomposition 
of Theorem \ref{thm:LDUweight}. 
We actually do not need the explicit expression for $\Phi^{(\nu)}$ to 
prove Theorem \ref{thm:Pearson}, but we give it for completeness in 
\eqref{eq:PhuPsiexplicit} since it gives a highly non-trivial example of the 
theory for Christoffel transformations for matrix-valued weights as presented in \cite{AlvaAGAMM}. 

\begin{proof}[Proof of Theorem \ref{thm:Pearson}.] 
In case $F_0=0$ in Lemma \ref{lem:symmetricD}, we can integrate the last symmetry 
condition using the last boundary condition to obtain 
$\frac{d(F_2W)}{dx}(x) = F_1(x)W(x)$. 
Note that 
$\mathscr{D}^{(\nu)}$ in \eqref{eq:relDPhiPsi-DnuEnu} has no constant term. 
$\mathscr{D}^{(\nu)}$ is expressed in terms of 
the symmetric differential operators $E^{(\nu)}$ and $D^{(\nu)}$, see 
Theorem \ref{thm:thmdifferentialoperators}, hence symmetric with respect to 
$W^{(\nu)}$. Noting that $F_2=(\Phi^{(\nu)})^\ast$, $F_1=(\Psi^{(\nu)})^\ast$, 
we find $\frac{d((\Phi^{(\nu)})^\ast W^{(\nu)})}{dx}(x) = (\Psi^{(\nu)})^\ast(x)W^{(\nu)}(x)$.
Taking adjoints proves the first identity of Theorem \ref{thm:Pearson}.

The proof of the second identity is more involved. We will show that 
\begin{equation}\label{eq:firstidentityforWPearson}
\frac{dW^{(\nu+1)}}{dx}(x) = c^{(\nu)} W^{(\nu)}(x)\Psi^{(\nu)}(x), 
\end{equation}
implying $\frac{dW^{(\nu+1)}}{dx}(x) =c^{(\nu)} \frac{d(W^{(\nu)}\Phi^{(\nu)})}{dx}(x)$. 
Integrating gives the second identity, since 
$W^{(\nu+1)}$ vanishes at $x=-1$ by Definition \ref{def:weightW} as does $W^{(\nu)}\Phi^{(\nu)}=(\Phi^{(\nu)})^\ast W^{(\nu)}$
by the first boundary condition of Lemma \ref{lem:symmetricD} applied to $\mathcal{D}^{(\nu)}$. 
\end{proof}

In order to prove \eqref{eq:firstidentityforWPearson} we use the LDU-decompositions 
of the weights involved. 
For this we need the relation between $L^{(\nu)}$ and $L^{(\nu+1)}$.

\begin{prop}\label{prop:relationLnuandLnu+1}
The matrix $L^{(\nu)}$ satisfies the following identities
\begin{gather*}
(L^{(\nu)}(x))^{-1} \,  L^{(\nu+1)}(x)= M_1^{(\nu)}(x),  \qquad 
\frac{d((L^{(\nu)})^{-1})}{dx}(x)\, L^{(\nu+1)}(x) = M_2^{(\nu)}(x),\\
M_1(x)=\Id +(1-x^2)\sum_{i=2}^{2\ell} \frac{i(i-1)}{(2\nu+2i-1)(2\nu+2i-3)}\, E_{i,i-2}, \\
M_2(x)=-\sum_{i=1}^{2\ell} i\, E_{i,i-1}-x\sum_{i=2}^{2\ell} \frac{i(i-1)}{(2\nu+2i-3)}\, E_{i,i-2}.
\end{gather*}
\end{prop}

We prove Proposition \ref{prop:relationLnuandLnu+1} in Appendix 
\ref{app:C-proofMVPearson}. 
Note that Proposition \ref{prop:relationLnuandLnu+1} implies 
% \begin{equation}
% \frac{dL^{(\nu)}}{dx}(x) = -L^{(\nu)}(x) \, M_2^{(\nu)}(x)(M_1^{(\nu)}(x))^{-1},\\
% \end{equation}
% using $\frac{dA^{-1}}{dx} = - A^{-1} \frac{dA}{dx}A^{-1}$ and $M^{(\nu)}_1$ being invertible
% for all $x$, as well as 
\begin{equation}\label{eq:fromProp:relationLnuandLnu+1}
(L^{(\nu)}(x))^{-1}\, \frac{dL^{(\nu+1)}}{dx}(x) = \frac{dM^{(\nu)}_1}{dx}(x)-M_2^{(\nu)}(x)
\end{equation}
by differentiation of the first identity of Proposition \ref{prop:relationLnuandLnu+1} and using the second identity.

Using the LDU-decomposition of Theorem \ref{thm:LDUweight} for $\nu$ and $\nu+1$, 
and Proposition \ref{prop:relationLnuandLnu+1} and \eqref{eq:fromProp:relationLnuandLnu+1} 
the identity \eqref{eq:firstidentityforWPearson} is equivalent to showing 
\begin{equation}\label{eq:firstidentityforWPearson1}
A^{(\nu)}(L^{(\nu)})^t =c^{(\nu)} (L^{(\nu)})^t \, \Psi^{(\nu)},
\end{equation}
where
\begin{multline*}
A^{(\nu)}(x) = (T^{(\nu)}(x))^{-1} \left[ \left(\frac{dM_1^{(\nu)}}{dx}(x)-M_2^{(\nu)}(x))\right)T^{(\nu+1)}(x)(M_1^{(\nu)}(x))^t + \right. \\ 
\left.
M_1^{(\nu)}(x)\frac{dT^{(\nu+1)}}{dx}(x) (M_1^{(\nu)}(x))^t + M_1^{(\nu)}(x)T^{(\nu+1)}(x)\left(\frac{dM_1^{(\nu)}}{dx}(x)-M_2^{(\nu)}(x)\right)^t\,\right]
\end{multline*}
The expression for $A^{(\nu)}$ might look complicated, but since $T^{(\nu)}$ is diagonal,
and $M_1^{(\nu)}$, $M_2^{(\nu)}$ only have two non-zero diagonals, $A^{(\nu)}$ is band-limited.
Using the expressions from Theorem \ref{thm:LDUweight} and Proposition 
\ref{prop:relationLnuandLnu+1} we find that $A^{(\nu)}$ is a three-diagonal 
matrix-valued polynomial of degree $2$; 
\begin{multline}\label{eq:Aexplicit}
2\nu(\nu+2\ell)(\nu+\ell) A^{(\nu)}(x)= 
\sum_{k=1}^{2\ell} (2\ell-i+1)(2\nu+1)(\nu+2\ell+1)(\nu+i-1)E_{i,i-1} \\
-x(2\nu+1)(\nu+2\ell+1) \sum_{k=0}^{2\ell} (2\nu^2+2\nu \ell+2\nu i+i(i-1)) E_{i,i} \\
+ (1-x^2) (2\nu+1)(2\ell+\nu+1) \sum_{k=0}^{2\ell-1}
\frac{(i+\nu)(i+1)(2\nu+i-1)(i+2\nu+2\ell)}{(2\nu+2i-1)(2\nu+2i+1)}E_{i,i+1}.
\end{multline}

In order to prove \eqref{eq:firstidentityforWPearson1} we need the explicit 
expression for $\Psi^{(\nu)}$, and for completeness we also write down 
the explicit expression for $\Phi^{(\nu)}$. 
The matrix-valued polynomials $\Phi^{(\nu)}$ and $\Psi^{(\nu)}$ are introduced in 
\eqref{eq:relDPhiPsi-DnuEnu}. A straightforward calculation using 
Theorem \ref{thm:thmdifferentialoperators} gives 
\begin{equation}\label{eq:PhuPsiexplicit}
\begin{split}
&\Phi^{(\nu)}(x)= \Phi(x) + \frac{(\ell+\nu)^2}{\ell^2}(1-x^2)\Id, \\
&\Phi(x) = x^{2}\sum_{i=0}^{2\ell}
\frac{\left( \ell-i\right)^{2}}{\ell^2} E_{i,i}+ 
 x\sum_{i=1}^{2\ell}\frac{(i-1-2\ell)(2\ell-2i+1)}{2 \ell^2} E_{i,i-1}+ \\ 
 &\qquad x\sum_{i=0}^{2\ell-1}\frac{(i+1)(2\ell-2i-1)}{2 \ell^2} E_{i,i+1} + 
\sum_{i=2}^{2\ell}\frac{(2\ell-i+2)(2\ell-i+1)}{4\ell^2} E_{i,i-2}+\\ 
&\qquad \sum_{i=0}^{2\ell}\frac{- i (2\ell-i+1)-(2\ell-i)(i+1)}{4 \ell^2} E_{i,i}  +
 \sum_{i=0}^{2\ell-2}\frac{(i+2)(i+1)}{4 \ell^2} E_{i,i+2}, \\
\end{split}
\end{equation}
and 
\begin{equation}\label{eq:PhuPsiexplicit2}
\begin{split}
& \qquad \frac{- \ell^2}{2\ell+2\nu+1}\Psi^{(\nu)}(x)=
x\sum_{i=0}^{2\ell} ( \nu+i) (\nu+2\ell-i) E_{i,i}+ \\
&\frac12 \sum_{i=1}^{2\ell} (i-1-2\ell) (\nu+i-1) E_{i,i-1} 
+ \frac12 \sum_{i=0}^{2\ell-1} (i+1) (\nu+2\ell-i-1) E_{i,i+1}. \\
\end{split}
\end{equation}

We prove \eqref{eq:firstidentityforWPearson1} in Appendix \ref{app:C-proofMVPearson} 
by showing that the matrix entries 
on both sides, each consisting of at most three Gegenbauer polynomials, are equal using
standard identities for Gegenbauer polynomials. Note that a straightforward check of 
\eqref{eq:firstidentityforWPearson} using Definition \ref{def:weightW} 
and Gegenbauer polynomials is much more involved, but this can also be done.

%%%%%%%%%%%%%%%%%%%%%%%%%%%%%%%%%%%%%%%%%%%%%%%%%%%%%%%%%%%%%%%%%%%
%%%%%SECTION%%%%%%%%%%%%%%%%%%%%%%%%%%%%%%%%%%%%%%%%%%%%%%%%%%%%%%%
%%%%%%%%%%%%%%%%%%%%%%%%%%%%%%%%%%%%%%%%%%%%%%%%%%%%%%%%%%%%%%%%%%%
\section{Matrix-valued orthogonal polynomials}\label{sec:MVOPs}

We prove all the statements of Section \ref{sec:MVGegenbauerpolsandprops} 
concerning the corresponding matrix-valued orthogonal polynomials. 
We also prove the remaining statements of Section  \ref{sec:differentialoperators}, i.e.
Proposition \ref{prop:commutantW} and the commutativity statement of 
Theorem \ref{thm:thmdifferentialoperators}. 

%%%%%SECTION%%%%%%%%%%%%%%%%%%%%%%%%%%%%%%%%%%%%%%%%%%%%%%%%%%%%%%%
%%%%%%%%%%%%%%%%%%%%%%%%%%%%%%%%%%%%%%%%%%%%%%%%%%%%%%%%%%%%%%%%%%%
\subsection{Rodrigues formula and the squared norm}\label{ssec:RodriguesDarboux}

\begin{proof}[Proof of Theorem \ref{thm:squarenormRodrigues}.] 
The leading coefficient $n\text{Id}$ of $\frac{dP^{(\nu)}_n}{dx}$ is non-singular, and for $k\in\N$ with $k<n-1$ we get
\begin{gather*}
\int_{-1}^1 \frac{dP^{(\nu)}_n}{dx}(x) W^{(\nu+1)}(x) x^k \, dx = \\
- c^{(\nu)} \int_{-1}^1 P^{(\nu)}_n(x) \Bigl( \frac{d(W^{(\nu)}\Phi^{(\nu)})}{dx}(x) x^k  + kW^{(\nu)}(x)\Phi^{(\nu)}(x) x^{k-1}\, \Bigr)dx 
\end{gather*}
using integration by parts and Theorem \ref{thm:Pearson}.
Consider the degrees using Theorem \ref{thm:Pearson} again, we see that the integral is zero for all $k<n-1$.
Hence $\frac{dP^{(\nu)}_n}{dx}(x) = n P_{n-1}^{(\nu+1)}(x)$, and (ii) follows.

By Corollary \ref{cor:adjointddx} and the factorisation $P\mathscr{D}^{(\nu)} = \left(\frac{dP}{dx}\right) S^{(\nu)}$ 
we see that $S^{(\nu)}$ maps $P_{n-1}^{(\nu+1)}$ to a multiple of $P_{n}^{(\nu)}$, and this multiple can be obtained
from considering leading coefficients, or by considering the eigenvalues of $\mathscr{D}^{(\nu)}$ obtained from 
\eqref{eq:relDPhiPsi-DnuEnu} and Theorem \ref{thm:thmdifferentialoperatorsP_d}. This gives 
\begin{equation*}
(P^{(\nu+1)}_{n-1}S^{(\nu)})(x) = K_n^{(\nu)} P^{(\nu)}_n(x), \quad
\frac{\ell^2}{2\ell+2\nu+n}\bigl(K_n^{(\nu)}\bigr)_{k,l} = - \de_{k,l} 
(\nu+k)(2\ell+\nu-k)
\end{equation*}
In particular, $S^{(\nu)}$ is a shift or raising operator. 

Now use Corollary \ref{cor:adjointddx} to see
\begin{gather*}
n \langle P^{(\nu+1)}_{n-1},P^{(\nu+1)}_{n-1}\rangle^{(\nu+1)} = 
\langle \frac{dP^{(\nu)}_{n}}{dx},P^{(\nu+1)}_{n-1} \rangle^{(\nu+1)} = 
- c^{(\nu)} \langle P^{(\nu)}_{n},P^{(\nu+1)}_{n-1}S^{(\nu)} \rangle^{(\nu)} = \\
- \langle P^{(\nu)}_{n}, P^{(\nu)}_{n}\rangle^{(\nu)}  
c^{(\nu)} \bigl(K_n^{(\nu)}\bigr)^\ast \qquad \Longrightarrow \\
H_n^{(\nu)} = \frac{-n}{c^{(\nu)}} H_{n-1}^{(\nu+1)} 
\bigl( (K_n^{(\nu)})^\ast\bigr)^{-1} =
(-1)^n n! H_0^{(\nu+n)} \Bigl(\prod_{i=0}^{n-1} c^{(\nu+i)} K_{n-i}^{(\nu+i)}\Bigr)^{-1}
\end{gather*}
since $K_n^{(\nu)}$ is a real diagonal matrix.  Note that 
the matrices commute, so we do not need to specify the order. 
It suffices to calculate $H_0^{(\nu)}$ for general $\nu>0$, and this follows 
immediately from the explicit expression of Definition \ref{def:weightW}
and the orthogonality relations of the (scalar) Gegenbauer polynomials
\eqref{eq:orthoscalarGegenbauer}. 
This leads to (i).

For (iii),  use Theorem \ref{thm:Pearson} to write  
$\bigl( QS^{(\nu)}\bigr)(x)= \bigl(c^{(\nu)}\bigr)^{-1}
\frac{d}{dx}\bigl( Q(x) W^{(\nu+1)}(x)\bigr) W^{(\nu)}(x)^{-1}$, which 
is a raising operator preserving polynomials. Iterating shows
\begin{equation*}
\left(\prod_{i=0}^{n-1} c^{(\nu+i)}\right)\ \Bigl( \bigl( (QS^{(\nu+n-1)}) \cdots S^{(\nu+1)}\bigr)S^{(\nu)}\Bigr)(x)  = 
  \frac{d^n}{dx^n}\Bigl( Q(x) W^{(\nu+n)}(x)\Bigr) \bigl(W^{(\nu)}(x)\bigr)^{-1}.
\end{equation*}
Now take $Q(x)= P^{(\nu+n)}_0(x) = \Id$, 
so that the left hand side equals $\prod_{i=0}^{n-1} 
c^{(\nu+i)}K^{(\nu+i)}_{n-i}\, P^{(\nu)}_n(x)$. 
A calculation gives the diagonal matrix $G_n^{(\nu)}$, and we obtain 
the Rodrigues formula (iii). 
\end{proof}

%%%%%SECTION%%%%%%%%%%%%%%%%%%%%%%%%%%%%%%%%%%%%%%%%%%%%%%%%%%%%%%%
%%%%%%%%%%%%%%%%%%%%%%%%%%%%%%%%%%%%%%%%%%%%%%%%%%%%%%%%%%%%%%%%%%%
\subsection{Eigenfunctions of differential operators
and scalar orthogonal polynomials}\label{ssec:eigenfunctionsMVOP}

\begin{proof}[Proof of Theorem \ref{thm:thmdifferentialoperatorsP_d}.]
Since the differential operators of Theorem \ref{thm:thmdifferentialoperators}
are symmetric with respect to the weight $W^{(\nu)}$ on $(-1,1)$ 
by Theorem \ref{thm:thmdifferentialoperators}, 
and since the differential operators also 
preserve the polynomials of fixed degree, we find that such a differential 
operator acting on $P^{(\nu)}_n$ is an orthogonal polynomial of degree $n$
with respect to the weight $W^{(\nu)}$ on $(-1,1)$. 
Hence it can be written as $\La_n P^{(\nu)}_n$, and  
the eigenvalue matrix $\La_n$ follows by considering leading 
coefficients. 
\end{proof}

\begin{proof}[Final part of the proof of Theorem \ref{thm:thmdifferentialoperators}.]
Theorem \ref{thm:thmdifferentialoperators} is proved in Section \ref{ssec:explicitdifferentialoperatorsconjugation}, except for the 
commutativity of $E^{(\nu)}$ and $D^{(\nu)}$. 
By Theorem \ref{thm:thmdifferentialoperatorsP_d} we see that 
$E^{(\nu)}$ and $D^{(\nu)}$ acting on the matrix-valued Gegenbauer polynomials
commute, since the diagonal eigenvalue matrices $\La_n(E^{(\nu)})$ and $\La_n(D^{(\nu)})$ 
commute. By \cite[Proposition 2.8]{GrunT} it follows that $E^{(\nu)}$ and $D^{(\nu)}$ commute. 
\end{proof}

We can now use the differential operators, especially the diagonal (or uncoupled) 
differential operator of Lemma \ref{lem:conjugatedDnuandEnu}, to obtain precise information on the 
matrix-valued orthogonal polynomials. In particular, we can derive
Theorem \ref{thm:calPnasRacahGegenbauer} in this way.
Indeed,  $P^{(\nu)}_n$ are eigenfunctions of 
$D^{(\nu)}-2\ell E^{(\nu)}$. Hence,  
Lemma \ref{lem:conjugatedDnuandEnu} shows that 
$\cR^{(\nu)}_n(x)\, = \, P^{(\nu)}_n(x)L^{(\nu)}(x)$ are eigenfunctions 
to a diagonal differential operator. 
Note that in general $\cR^{(\nu)}_n$ 
is a matrix-valued polynomial of degree $n+2\ell$ with 
highly singular leading coefficient. Because of Lemma \ref{lem:conjugatedDnuandEnu} the
matrix entry $\bigl(\cR^{(\nu)}_n(x)\bigr)_{k,j}$  is a polynomial solution to the hypergeometric differential operator 
\[
(1-x^2)f''(x) -x(2\nu+2j+1)f'(x) +(n+k-j)(j+n+k+2\nu)f(u)=0.
\]
Since the polynomial solutions are unique up to a constant we find 
\begin{equation}\label{eq:calRasGegenbauer}
\bigl(\cR^{(\nu)}_n(x)\bigr)_{k,j}\, =\, c^{(\nu)}_{k,j}(n)\, C^{(\nu+j)}_{n+k-j}(x).
\end{equation}
Similarly, by Theorem \ref{thm:thmdifferentialoperatorsP_d} and Lemma \ref{lem:conjugatedDnuandEnu}, we have that 
$\cR^{(\nu)}_n (E^{(\nu)})\widetilde{\ } =  \La_n(E^{(\nu)})\cR^{(\nu)}_n$. Note that the eigenvalue does not change.
Using the explicit expression of Lemma \ref{lem:conjugatedDnuandEnu} and that $\La_n(E^{(\nu)})$ is diagonal, plugging $x=1$ gives the recurrence 
\begin{multline}
\label{eq:recurrence_ck}
(2\nu+2i-2)(2\ell-i+1)c^{(\nu)}_{k,i-1}(n) + (i(i+2\nu-1)-4\ell(\ell+1)-2\ell (\nu-1))c^{(\nu)}_{k,i}(n)\\
+ \frac{(i+1)(2\nu+i-1)(2\nu+2\ell+i)(n+k-i)(n+k+i+2\nu)}{(2\nu+2i-1)(2\nu+2i)(2\nu+2i+1)} c^{(\nu)}_{k,i+1}(n)\\
  =\, (-2n(\ell-k)+(2\ell+2)(k-2\ell)-2(\nu-1)(\ell-k))c^{(\nu)}_{k,i}(n), 
\end{multline}
which can be solved in terms of Racah polynomials as in \cite{KoelvPR2}, see also \cite{PachTZ}; 
\begin{equation}\label{eq:coeff_c_as_Racah}
\begin{split}
c^{(\nu)}_{k,i} &= c^{(\nu)}_{k,0}\, (-1)^i\frac{(-2\ell)_i\, (-n-k)_i\, (n+k-i)! }{i!\, (2\nu+2\ell)_i \, (2\nu+2i)_{n+k-i}}\, 
R_k(\la(i);-2\ell-1,-n-k-\nu,\nu-1,\nu-1) \\ &= 
c^{(\nu)}_{k,0}\, (-1)^i\frac{(-2\ell)_i\, (-n-k)_i \, (n+k-i)! }{i!\, (2\nu+2\ell)_i \, (2\nu+2i)_{n+k-i}}\, 
\rFs{4}{3}{-i,i+2\nu-1,-k, -n-\nu-2\ell}{\nu,-n-k,-2\ell}{1}.
\end{split}
\end{equation} 

Since the inverse of the lower triangular matrix $L^{(\nu)}$ has been calculated explicitly by 
Cagliero and Koornwinder \cite[Thm. 4.1]{CaglK} as 
\begin{equation}\label{eq:inverseLbyCK}
\Bigl( L^{(\nu)}(x)^{-1}\Bigr)_{i,j} = \frac{i!}{j! \, (2\nu+i+j-1)_{i-j}}\, C^{(1-\nu-i)}_{i-j}(x), \qquad i\geq j,
\end{equation}
the proof of Theorem \ref{thm:calPnasRacahGegenbauer} follows directly from
\eqref{eq:calRasGegenbauer} and \eqref{eq:coeff_c_as_Racah} up to an explicit expression of 
$c^{(\nu)}_{k,0}$.

It remains to compute $c^{(\nu)}_{k,0}$. If we combine $\mathcal{R}^{(\nu)}_n=P_n^{(\nu)}L^{(\nu)}$ with Theorem \ref{thm:squarenormRodrigues}(ii), we obtain
\begin{equation}
\label{eq:Rn_M1M2}
n\mathcal{R}^{(\nu+1)}_{n-1}(x)= \left(\frac{d \mathcal{R}^{(\nu)}_n(x)}{dx}\right) \, M_1^{(\nu)}(x)+ \mathcal{R}^{(\nu)}_n(x)\, M_2^{(\nu)}(x).
\end{equation}
It follows from \eqref{eq:calRasGegenbauer} and Proposition \ref{prop:relationLnuandLnu+1} that the $(k,0)$-th entry of 
\eqref{eq:Rn_M1M2}, evaluated at $x=1$, is given by
\begin{equation}
\label{eq:recurrence_c_M1M2}
n c_{k,0}^{(\nu+1)}(n-1) = 2\nu \, c_{k,0}^{(\nu)}(n) - c_{k,1}^{(\nu)}(n) -\frac{2(n+k-1)(n+k+2\nu+1)}{(2\nu+1)(2\nu+2)(2\nu+3)} \, c_{k,2}^{(\nu)}.
\end{equation}
If we write $c_{k,2}^{(\nu)}(n)$ in terms of $c_{k,0}^{(\nu)}(n)$ and $c_{k,1}^{(\nu)}(n)$ using the recurrence \eqref{eq:coeff_c_as_Racah} for $i=1$, and we replace in \eqref{eq:recurrence_c_M1M2}, we obtain
\begin{equation}
\label{eq:recurrence_c_M1M2_2}
n c_{k,0}^{(\nu+1)}(n-1) = \left(2\nu +\frac{2\ell}{(2\nu+2\ell+1)}\right) \, c_{k,0}^{(\nu)}(n) - \left(1-\frac{(2\nu-k(2\nu+2n+2\ell)+2\ell n)}{2\nu(2\nu+2\ell+1)}\right) c_{k,1}^{(\nu)}(n) .
\end{equation}
Finally, from the recurrence \eqref{eq:coeff_c_as_Racah} for $i=0$ we write $c_{k,1}^{(\nu)}(n)$ in terms of $c_{k,0}^{(\nu)}(n)$. If we replace this in \eqref{eq:recurrence_c_M1M2_2} we get
\begin{equation}
\label{eq:recurrence_c_two_terms}
c_{k,0}^{(\nu+1)}(n-1) = - \frac{2(2\nu+1)(\nu+k)(k-\nu-2\ell)(2\nu+n+2\ell)}{(2\nu+2\ell+1)(\nu+\ell)(n+k)(n+k+2\nu)} \, c_{k,0}^{(\nu)}(n).
\end{equation}
Moreover, since $\mathcal{R}_0^{(\nu)}=L^{(\nu)}$, we have $c_{k,0}^{(\nu)}(0)=(2\nu)_k$ and therefore the recurrence \eqref{eq:recurrence_c_two_terms} is solved explicitly by
\[
c_{k,0}^{(\nu)}(n)=\frac{2^{-n}(\nu)_n (2\ell+2\nu)_n (n+k)!}{(\nu+k)_n (2\ell+\nu-k)_n(2\nu)_{n+k}}.
\]
Observe that for the case $\nu=1$, the initial value $c_{k,0}^{(\nu)}$ differs from that in \cite{KoelvPR2} by a factor
$(-1)^n 2^{-n} (n+k)!/(2\nu)_{n+k}$. This is due to a different normalisation in 
\eqref{eq:calRasGegenbauer} and the fact the orthogonality interval is $[-1,1]$ 
and $[0,1]$ in \cite[\S 6]{KoelvPR2}.

%%%%%SECTION%%%%%%%%%%%%%%%%%%%%%%%%%%%%%%%%%%%%%%%%%%%%%%%%%%%%%%%
%%%%%%%%%%%%%%%%%%%%%%%%%%%%%%%%%%%%%%%%%%%%%%%%%%%%%%%%%%%%%%%%%%%
\subsection{Three-term recurrence relation}\label{ssec:3termrecurrenceMVOP}

To calculate the coefficients $B^{(\nu)}_n$, $C^{(\nu)}_n$ in Theorem 
\ref{thm:three_term_for_Pn} we first note that $C^{(\nu)}_n$ 
follows from the squared norm; 
\begin{equation*}
C_n^{(\nu)} =  H^{(\nu)}_n \bigl( H^{(\nu)}_{n-1}\bigr)^{-1}
\end{equation*}
and the explicit expression of Theorem \ref{thm:squarenormRodrigues}. 
 
Writing $P_n^{(\nu)}(x) = x^n\Id + x^{n-1} X^{(\nu)}_n + \cdots$, we find
\begin{equation*}
B_n^{(\nu)} =  X^{(\nu)}_n - X^{(\nu)}_{n+1}
\end{equation*}
by comparing coefficients of $x^{n}$ in the three-term recurrence relation.
Now Theorem \ref{thm:three_term_for_Pn} follows from these 
observations and Lemma \ref{lem:onebutleadingcoefficient}.

\begin{lem}\label{lem:onebutleadingcoefficient}
The one-but-leading coefficient $X^{(\nu)}_n$ is given by
\begin{gather*}
\frac{n}{2} \sum_{i=0}^{2\ell-1} \frac{i-2\ell}{2\ell+\nu+n-1-i} E_{i,i+1}
- \frac{n}{2} \sum_{i=1}^{2\ell} \frac{i}{\nu+n-1+i} E_{i,i-1}
\end{gather*}
\end{lem}

\begin{proof} Differentiate $P_n^{(\nu)}$ 
and use Theorem \ref{thm:squarenormRodrigues}(ii) to find 
$(n-1) X^{(\nu)}_n =  n X^{(\nu+1)}_{n-1}$. 
Iterating gives $X^{(\nu)}_n =  n X^{(\nu+n-1)}_{1}$, so that it suffices the 
check the polynomial of degree $1$. Using the Rodrigues formula of 
Theorem \ref{thm:squarenormRodrigues}(iii) and the adjoint of the relations 
in Theorem \ref{thm:Pearson}, we find $P_1^{(\nu)}(x)= c^{(\nu)} G_1^{(\nu)} 
\bigl( \Psi^{(\nu)}(x)\bigr)^\ast$. Using \eqref{eq:PhuPsiexplicit2} we find  
$X^{(\nu)}_1$,
and hence $X^{(\nu)}_n$. 
\end{proof}

%%%%%SECTION%%%%%%%%%%%%%%%%%%%%%%%%%%%%%%%%%%%%%%%%%%%%%%%%%%%%%%%
%%%%%%%%%%%%%%%%%%%%%%%%%%%%%%%%%%%%%%%%%%%%%%%%%%%%%%%%%%%%%%%%%%%
\subsection{Commutant}\label{ssec:commutant}
The commutant algebra $A^{(\nu)}$ in Proposition \ref{prop:commutantW} is a 
$\ast$-algebra, hence generated by its self-adjoint elements. Let
$T=T^\ast \in A^{(\nu)}$ be an invertible self-adjoint element, 
then $P\mapsto PT$ is a symmetric operator with 
respect to $W^{(\nu)}$, a matrix-valued differential operator of 
order $0$.  Since it preserves the polynomials, the 
matrix-valued Gegenbauer polynomials are eigenfunctions.
Hence, $P_n^{(\nu)}(x) T = T P_n^{(\nu)}(x)$ for all $n$ and all $x$ by comparing leading coefficients,
cf. \cite[Lem.~3.1(2)]{KoelR}. In particular, $T$ commutes with the one-but-leading coefficient of $P_n^{(\nu)}$, i.e.  $X_n^{(\nu)} T=T  X_n^{(\nu)} $ for all $n\in \N$. From the explicit expression of $X^{(\nu)}_n$ in Lemma \ref{lem:onebutleadingcoefficient}, we find that $T_{i,j}=0$ unless $j=i$ or $j=2\ell-i$, $T_{i,i}=T_{i+1,i+1}$ and $T_{i,2\ell-i}=T_{i+1,2\ell-i-1}$ for all $i=0,\ldots,2\ell-1$. This implies that $T\in \langle\{I,J\}\rangle$.

Observe that $W^{(\nu)}(x)J=JW^{(\nu)}(x)$ for all $x\in(-1,1)$ if and only if $(W^{(\nu)}(x))_{2\ell -n,m} =(W^{(\nu)}(x))_{n,2\ell-m}$ for all
$n,m$ and all $x$. This means that the weight matrix $W^{(\nu)}(x)$ is persymmetric (or orthosymmetric).  
Using Definition \ref{def:weightW} and comparing coefficients of the Gegenbauer polynomials we need to prove 
\begin{equation*}
\al_t^{(\nu)} \bigl(\min(2\ell-m,n),\max(2\ell-m,n)\bigr) = \al_{t+m-n}^{(\nu)}\bigl(\min(2\ell-n,m), \max(2\ell-n,m)\bigr).
\end{equation*}
This is straightforwardly verified from the expression in Definition \ref{def:weightW}.
%%%%%%%%%%%%%%%%%%%%%%%%%%%%%%%%%%%%%%%%%%%%%%%%%%%%%%%%%%%%%%%%%%%
%%%%%APPENDIX%%%%%%%%%%%%%%%%%%%%%%%%%%%%%%%%%%%%%%%%%%%%%%%%%%%%%%
%%%%%%%%%%%%%%%%%%%%%%%%%%%%%%%%%%%%%%%%%%%%%%%%%%%%%%%%%%%%%%%%%%%
\appendix

%%%%%%%%%%%%%%%%%%%%%%%%%%%%%%%%%%%%%%%%%%%%%%%%%%%%%%%%%%%%%%%%%%%
%%%%%SECTION%%%%%%%%%%%%%%%%%%%%%%%%%%%%%%%%%%%%%%%%%%%%%%%%%%%%%%%
%%%%%%%%%%%%%%%%%%%%%%%%%%%%%%%%%%%%%%%%%%%%%%%%%%%%%%%%%%%%%%%%%%%
\section{Proof of the LDU-decomposition}\label{app:proofthmLDUdecompW}

In Appendix \ref{app:proofthmLDUdecompW} we prove Theorem \ref{thm:LDUweight}. 
The proof is an extension of the result \cite[Thm.~2.1, App.~A]{KoelvPR2} for the
case $\nu=1$. Note that the proof is verificational; it has initially been obtained by
use of computer algebra for specific values of $\nu$ and $\ell$. We indicate all the steps
and we leave out the manipulation of shifted factorials. 
We assume $\nu>0$, but it is possible to extend to argument to  
$\nu>-\frac12$, $\nu\not=0$.

By using the symmetry and 
taking matrix elements we see that the LDU-decomposition is equivalent to proving
\begin{equation}\label{eq:equavalenttoThmLDU}
\sum_{t=0}^m %\hskip-.3truecm
\al^{(\nu)}_t(m,n) \, C^{(\nu)}_{m+n-2t}(x) \, = \,
\sum_{k=0}^m \frac{m!\, n!\, t^{(\nu)}_k\, (1-x^2)^k}{k!\, k!\, (2\nu+2k)_{m-k}(2\nu+2k)_{n-k}}
C^{(\nu+k)}_{m-k}(x)\, C^{(\nu+k)}_{n-k} 
\end{equation}
for $n\geq m$. Note first that the right hand side is a polynomial of degree $n+m$.
We need to find its expansion in terms of Gegenbauer polynomials $C^{(\nu)}_r(x)$, so that we 
need the integral
\begin{equation}\label{eq:integralof3Gegenbauers}
\int_{-1}^1 C^{(\nu+k)}_{m-k}(x)\, C^{(\nu+k)}_{n-k}(x) C^{(\nu)}_r(x)\, (1-x^2)^{k+\nu-1/2}\, dx.
\end{equation}
Since the Gegenbauer polynomials are symmetric and by the orthogonality relations for the 
Gegenbauer polynomials $C^{(\nu+k)}_n$ we see that 
this integral is zero for $n+m-r$ odd and for $r<n-m$, so that the expansion of the 
right hand side of \eqref{eq:equavalenttoThmLDU} in terms of Gegenbauer polynomials $C^{(\nu)}_r(x)$
only has non-zero terms for the summands in the left hand side of \eqref{eq:equavalenttoThmLDU}.
Using \eqref{eq:orthoscalarGegenbauer} it
remains to prove the identity 
\begin{equation}\label{eq:equavalenttoThmLDU2}
\begin{split}
&\qquad\qquad\qquad 
\al^{(\nu)}_t(m,n) \frac{(2\nu)_{m+n-2t} \sqrt{\pi}\, \Ga(\nu+\frac12)}{(m+n-2t)! (m+n-2t+\nu) \Ga(\nu)} 
= \\
&\sum_{k=0}^m\frac{m!\, n!\, t^{(\nu)}_k}{k!\, k!\, (2\nu+2k)_{m-k}(2\nu+2k)_{n-k}} 
\int_{-1}^1 C^{(\nu+k)}_{m-k}(x)\, C^{(\nu+k)}_{n-k}(x) C^{(\nu)}_{m+n-2t}(x)\, (1-x^2)^{k+\nu-1/2}\, dx
\end{split}
\end{equation}
for $n\geq m$, $t\in\{0,\cdots,m\}$. 

The integral in \eqref{eq:equavalenttoThmLDU2} has been evaluated in \cite[Rmk.~2.8]{KoelvPR2} and 
in order to prove \eqref{eq:equavalenttoThmLDU2} we follow the steps in \cite[App.~A]{KoelvPR2}. The integral
in \eqref{eq:equavalenttoThmLDU2} is equal to a balanced ${}_4F_3$-series \cite[Rmk.~2.8]{KoelvPR2}, 
\begin{equation*}
\begin{split}
&\frac{(\nu+k)_{m-k}(2\nu+2k)_{n-k}(-m)_{m-t}(\nu)_{n-t}\Ga(\nu+k+1/2)}{(m-k)!\, (n-k)!\, (m-t)!\, \Ga(\nu +m+n-t+1)}\sqrt{\pi} \\
& \quad \times \rFs{4}{3}{k-m,-m-k-2\nu+1,t-m,\nu+1+n-t}{-m,-m-\nu+1,n-m+1}{1}
\end{split}
\end{equation*}
Using Whipple's transformation, see e.g. \cite[Thm.~3.3]{AndrAR}, twice (once with $(n,a,d)$ of
\cite[Thm.~3.3]{AndrAR} as $(m-k,-m-2\nu-k+1,-m)$ and the second time as $(t,-m-n-\nu+t,-m)$) the 
${}_4F_3$-series can be rewritten as 
\begin{equation}\label{eq:appA1}
\frac{(1-n-m-2\nu)_{m-k} (-n)_t (\nu)_t}{(1+n-m)_{m-k} (1-\nu-m)_t(1-m-n-2\nu)_t}
\rFs{4}{3}{-k,k+2\nu-1,-t,t-m-n-\nu}{-m,\ -n, \ \nu }{1}
\end{equation}

We now plug the obtained expression for the integral into the right hand side of \eqref{eq:equavalenttoThmLDU2}, where we replace the 
${}_4F_3$-series of \eqref{eq:appA1} by its sum $\sum_{j=0}^k$. We next interchange the summations over $k$ and $j$, and in the 
summation $\sum_{k=j}^m$ we replace $k=p+j$. This gives for the inner sum (the $k$-dependent part)
\begin{equation*}
\begin{split}
& \frac{(2\nu+2\ell)_j (1-n-m-2\nu)_{m-j}(2\nu+j)_j}{(-2\ell)_j (2\nu+j)_m (m-j)!\, (1+n-m)_{m-j}}  \\
& \qquad \times \rFs{5}{4}{2\nu+2j-1,\nu+j+\frac12, j-n, 2\nu+2l+j,j-m}{\nu+j-\frac12, 2\nu+j+n, -2\ell+j,2\nu+m+j}{1} \\
= & \frac{(1-n-m-2\nu)_m (n-2\ell)_m}{(-2\ell)_m m!\, (1+n-m)_m (2\nu+n)_m} 
\frac{(2\nu+2\ell)_j (-m)_j (-n)_j}{(1-n-m+2\ell)_j}
\end{split}
\end{equation*}
using the Dougall summation formula for a very-well-poised ${}_5F_4$-series, see e.g. \cite[Cor.~4.3]{AndrAR}, \cite[\S 4.3(3)]{Bail}.
This shows that the right hand side of \eqref{eq:equavalenttoThmLDU2} can be written as a single sum; explicitly 
\begin{equation*}
\begin{split}
&\frac{(-m)_{m-t}(\nu)_{n-t} \sqrt{\pi} (-n)_t (\nu)_t n!\, (2\ell+\nu)(\nu)_m \Ga(\nu+\frac12) (1-n-m-2\nu)_m (n-2\ell)_m}
{(n-m)!\, (m-t)!\, \Ga(\nu+n+m-t+1)(1-\nu-m)_t (1-n-m-2\nu)_t (-2\ell)_m} \\
& \qquad\times \frac{1}{ (1+n-m)_m (2\nu+n)_m}\, \rFs{3}{2}{-t, t-m-n-\nu, 2\nu+2\ell}{\nu, 1-n-m+2\ell}{1}
\end{split}
\end{equation*}
The balanced ${}_3F_2$-series is summable 
to $\frac{(1-m-n-2\nu)_t (-2\ell-\nu)_t}{(\nu)_t (2\ell+1-m-n)_t}$ by the Pfaff-Saalsch\"utz formula, see e.g. \cite[Thm.~2.2.6]{AndrAR}, 
\cite[\S 2.2]{Bail}, \cite[(1.4.5)]{Isma}.
Next a straightforward verification using the expression of Definition \ref{def:weightW} shows that this is equal to the 
left hand side of \eqref{eq:equavalenttoThmLDU2}, which proves Theorem \ref{thm:LDUweight}. 

%%%%%%%%%%%%%%%%%%%%%%%%%%%%%%%%%%%%%%%%%%%%%%%%%%%%%%%%%%%%%%%%%%%%%%
%%%%%%%%%%%%%%%%%%%%%%%%%%%%%%%%%%%%%%%%%%%%%%%%%%%%%%%%%%%%%%%%%%%%%%
%%%%%%%%%%%%%%%%%%%%%%%%%%%%%%%%%%%%%%%%%%%%%%%%%%%%%%%%%%%%%%%%%%%%%%

\section{Proof of Lemma \ref{lem:conjugatedDnuandEnu}}
\label{app:Bproofoflem:conjugatedDnuandEnu}
The following properties of the Gegenbauer polynomials
\begin{align}
\label{eq:Gegenbauer_three_term}
xC_i^{(\lambda)}(x)&=\frac{(i+1)}{2(i+\lambda)}C_{i+1}^{(\lambda)}(x)+\frac{(i+2\lambda-1)}{2(i+\lambda)}C_{i-1}^{(\lambda)}(x),\\
\label{eq:Gegenbauer_lambda_p1}
(i+\lambda)C_i^{(\lambda)}(x)&=\lambda(C_i^{(\lambda+1)}(x)-C_{i-2}^{(\lambda+1)}(x)),\\
\label{eq:Gegenbauer_1mx2}
(1-x^2)C_i^{(\lambda+1)}(x)&=-\frac{ (i+1)(i+2)}{4\lambda(i+\lambda+1)} 
C_{i+2}^{(\lambda)}(x)+\frac{ (i+2\lambda)(i+2\lambda+1)}{4\lambda(i+\lambda+1)}C_i^{(\lambda)}(x), 
\end{align}
are useful, see e.g. \cite[\S 4.5]{Isma}. We follow the convention that polynomials of negative degree are $0$. 

First we observe that $D^{(\nu)}-2\ell E^{(\nu)}=L^{(\nu)}(D^{(\nu)}-2\ell E^{(\nu)})\widetilde{\ }(L^{(\nu)})^{-1}$ if and only if
\begin{align}
\label{eq:symmetryDt_1}
2(1-x^2)\frac{dL^{(\nu)}}{dx}(x)+L^{(\nu)}(x)K_1(x)=(C-x(2\ell+2\nu+1)-2\ell(xB_1+B_0))L^{(\nu)}(x),\\
(1-x^2)\frac{d^2 L^{(\nu)}}{dx^2}(x)+\frac{d L^{(\nu)}}{dx}(x)K_1(x)+L^{(\nu)}(x)K_0  = (-V-2\ell A_0)L^{(\nu)}(x).
\label{eq:symmetryDt_2}
\end{align}
If we multiply the $(m,k)$-th entry of \eqref{eq:symmetryDt_2} by $m!/(k!(2\nu+2k)_{m-k})$ we obtain
\begin{multline}
\label{eq:diff_eq_gegen1}
(1-x^2) \frac{d^2 C^{(\nu+k)}_{m-k}}{dx^2}(x)
-x(2\nu+2k+1) \frac{d C^{(\nu+k)}_{m-k}}{dx}(x) \\
+ ( -k(2\nu+k)+2\ell(\nu+2\ell+1))C^{(\nu+k)}_{m-k}(x) \\
= ( m(2\ell-m)-(2\ell+2)(m-2\ell)+2(\nu-1)(\ell-m)) C^{(\nu+k)}_{m-k}(x).
\end{multline}
Since $ -k(2\nu+k)+2\ell(\nu+2\ell+1)-m(2\ell-m)+(2\ell+2)(m-2\ell)-2(\nu-1)(\ell-m) =(m-k)(m+k+2\nu)$, \eqref{eq:diff_eq_gegen1} is the differential equation for the Gegenbauer polynomial $C^{(\nu+k)}_{m-k}$, see \cite[\S 4.5]{Isma}, and hence \eqref{eq:symmetryDt_2} holds true.

On the other hand, if we multiply \eqref{eq:symmetryDt_1} by $m!/(k!(2\nu+2k)_{m-k})$, we obtain
\begin{multline}
2(1-x^2)\frac{dC^{(\nu+k)}_{m-k}}{dx}(x) - x(2\nu+2k+1)C^{(\nu+k)}_{m-k}(x) \\
=- x(2m+2\nu+1)C^{(\nu+k)}_{m-k}(x) + 2(2\nu+m+k-1) C^{(\nu+k)}_{m-1-k}(x).
\end{multline}
which can be verified using \cite[(22.7.21)]{Abr}.

In order to prove the statement for $(E^{(\nu)} )\widetilde{\ }$ we observe that $E^{(\nu)}=L^{(\nu)}(E^{(\nu)})\widetilde{\ }(L^{(\nu)})^{-1}$ if and only 
\begin{align}
\label{eq:conjugation_E1}
L^{(\nu)}S_1^{(\nu)}=(xB_1+B_0)L^{(\nu)},\qquad
(L^{(\nu)})'S_1^{(\nu)}+L^{(\nu)}S_0^{(\nu)}=A_0^{(\nu)}L^{(\nu)}.
\end{align}

If we multiply the first equation in \eqref{eq:conjugation_E1} by  $m!/(k!(2\nu+2k)_{m-k})$, we obtain
\begin{multline}
\label{eq:conjugation_E2}
\frac{k(2\ell-k+1)(2\nu+m+k-1)}{4\ell(2\nu+2k-1)(\nu+k-1)}\, C^{(\nu+k-1)}_{m-k+1}(x) \\
- \frac{(2\ell+2\nu+k)(2\nu+k-1)(\nu+k)}{\ell(2\nu+2k-1)(2\nu+k+m)}\, (1-x^2)\, C^{(\nu+k+1)}_{m-k-1}(x)
= -\frac{(\ell-m)}{\ell} \, x \, C^{(\nu+k)}_{m-k}(x) \\
 - \frac{(2\nu+m+k-1)}{2\ell} \, C^{(\nu+k)}_{m-k-1}(x)
+\frac{(m+1)(2\ell-m)}{2\ell(m+k+2\nu)} \, C^{(\nu+k)}_{m-k+1}(x).
\end{multline}
On the left hand side of \eqref{eq:conjugation_E2} we apply \eqref{eq:Gegenbauer_lambda_p1} on the first term and \eqref{eq:Gegenbauer_1mx2} on the second term and on the right hand side of \eqref{eq:conjugation_E2} we apply the three-term recurrence relation \eqref{eq:Gegenbauer_three_term}. Now \eqref{eq:conjugation_E2} only involves Gegenbauer polynomials  $C^{(\nu+k)}_{m-k-1}$ and $C^{(\nu+k)}_{m-k+1}$ and we verify \eqref{eq:conjugation_E2} by showing that the coefficients for each of these polynomials are equal.

If we multiply the second equation in \eqref{eq:conjugation_E1} by  $m!/(k!(2\nu+2k)_{m-k})$, we obtain
\begin{multline}
\label{eq:conjugation_E2_2}
\frac{k(2\nu+m+k-1)(2\ell-k+1)}{4\ell(\nu+k-1)(2\nu+2k-1)} \frac{dC^{(\nu+k-1)}_{m-k+1}}{dx} \, (x)  \\
-\frac{(\nu+k)(2\nu + k-1)(2\nu+k+2\ell)}{\ell(m+k+2\nu)(2\nu+2k-1)}\, (1-x^2) \, \frac{dC^{(\nu+k+1)}_{m-k-1}}{dx} \, (x) \\
+\frac{(k^2-k+2\nu k-4\ell^2-2\ell-2\nu\ell)}{2\ell} \, C^{(\nu+k)}_{m-k}(x) \\
+\frac{(\nu+k)(2\nu+2k+1)(k-1+2\nu)(2\nu+k+2\ell)}{\ell(m+k+2\nu)(2\nu+2k-1)} \, x \, C^{(\nu+k+1)}_{m-k-1}(x) \\
=-\frac{(2\ell^2-m\ell+\ell+\nu\ell-\nu m)}{\ell} C^{(\nu+k)}_{m-k}(x).
\end{multline}
In order to verify \eqref{eq:conjugation_E2_2} we proceed as before. We use \eqref{eq:Gegenbauer_lambda_p1}, \eqref{eq:Gegenbauer_1mx2} and \eqref{eq:Gegenbauer_three_term} to write \eqref{eq:conjugation_E2_2} as a combination of Gegenbauer polynomials with parameter $\nu+k$ and we verify that the coefficients for the coefficients on either side are equal. 

The $k$-th diagonal entry of $(D^{(\nu)}-2\ell E^{(\nu)})\widetilde{\ }$ is, up to a constant, the differential operator for the Gegenbauer polynomials $C^{(\nu+k)}_n$ and hence it is symmetric with respect to the $k$-th diagonal entry of $T^{(\nu)}$ being the weight for the Gegenbauer polynomials with parameter $\nu+k$. This implies the symmetry of $(D^{(\nu)}-2\ell E^{(\nu)})\widetilde{\ }$.

We prove the symmetry of the operator $(E)\widetilde{\ }$ by showing that the symmetry and boundary conditions in Lemma \ref{lem:symmetricD} hold true, taking $F_2=0$. The boundary conditions follow directly from the explicit expressions of $T^{(\nu)}$ in Theorem \ref{thm:LDUweight} and of $S_1^{(\nu)}$ in Lemma \ref{lem:conjugatedDnuandEnu}. Finally we need to prove
\begin{align}
\label{eq:symmetry_E_1}
-S_1^{(\nu)}T^{(\nu)} = T^{(\nu)}(S_1^{(\nu)})^*,\qquad - (S_1^{(\nu)}T^{(\nu)})'+S_0^{(\nu)}T^{(\nu)} = T^{(\nu)}(S_0^{(\nu)})^*.
\end{align}
Writing down the $(k,k-1)$-th and $(k,k+1)$-th entries of the first equation of \eqref{eq:symmetry_E_1} we see that it is equivalent to $-(T^{(\nu)}(x))_{k-1,k-1}/(T^{(\nu)}(x))_{k,k} = (S^{(\nu)}_1)_{k-1,k}/ (S^{(\nu)}_1)_{k,k+1}$ which can be verified from the explicit expressions of $T^{(\nu)}$ and of $S_1^{(\nu)}$. On the other hand, the only non-zero entries of the second equation in \eqref{eq:symmetry_E_1} are the $(k,k-1)$-th and $(k-1,k)$-th entries. The verification is straightforward from the explicit expressions of $T^{(\nu)}$, $S_1^{(\nu)}$ and $S_0^{(\nu)}$.
\qed

%%%%%%%%%%%%%%%%%%%%%%%%%%%%%%%%%%%%%%%%%%%%%%%%%%%%%%%%%%%%%%%%%%%%%%
%%%%%%%%%%%%%%%%%%%%%%%%%%%%%%%%%%%%%%%%%%%%%%%%%%%%%%%%%%%%%%%%%%%%%%
%%%%%%%%%%%%%%%%%%%%%%%%%%%%%%%%%%%%%%%%%%%%%%%%%%%%%%%%%%%%%%%%%%%%%%

\section{Proof of the matrix-valued Pearson equation}
\label{app:C-proofMVPearson}

\begin{proof}[Proof of Proposition \ref{prop:relationLnuandLnu+1}.]
The first statement of Proposition \ref{prop:relationLnuandLnu+1} is equivalent to 
$L^{(\nu+1)}(x)-L^{(\nu)}(x)M_1^{(\nu)}(x)=0$.
The $(n,m)$-entry, divided by $m!/(n!(2\nu+2n+2)_{m-n})$, of the left hand side is given by
\begin{multline}
\label{eq:nm_identities_L_1v2}
C^{(\nu+n+1)}_{m-n}(x)-\frac{(2\nu+n+m)(2\nu+n+m+1)}{(\nu+n)(2\nu+2n+1)}C^{(\nu+n)}_{m-n}(x)\\
+\frac{2(\nu+n+1)}{(2\nu+2n+1)} (1-x^2)C^{(\nu+n+2)}_{m-n-2}(x).
\end{multline}
Use \eqref{eq:Gegenbauer_lambda_p1} in the second term of \eqref{eq:nm_identities_L_1v2} and
\eqref{eq:Gegenbauer_1mx2} in the third term of \eqref{eq:nm_identities_L_1v2}  to write
in terms of two Gegenbauer polynomials of parameter $\nu+n+1$ and degree $m-n$ and $m-n-2$.
By an easy calculation both coefficients are zero, hence the first statement of Proposition \ref{prop:relationLnuandLnu+1} follows.

Using $\frac{dA^{-1}}{dx} = -A^{-1} \frac{dA}{dx} A^{-1}$ and the first statement of 
Proposition \ref{prop:relationLnuandLnu+1} it suffices to prove 
\begin{equation}\label{eq:identities_L_1v2}
\frac{dL^{(\nu)}}{dx}(x)\, M_1^{(\nu)}(x) + L^{(\nu)}(x)M_2^{(\nu)}(x) = 0.
\end{equation}
The $(n,m)$-entry of the left hand side of \eqref{eq:identities_L_1v2}, divided by $m!/(n!(2\nu+2n)_{m-n})$, is given by
\begin{multline}
\label{eq:nm_identities_L_2v2}
2(\nu+n)C^{(\nu+n+1)}_{m-n-1}(x) + \frac{8(\nu+n)(\nu+n+1)(\nu+n+2)}{(m+n+2\nu)(m+n+2\nu+1)} (1-x^2)C_{m-n-3}^{(\nu+n+3)}(x) \\
-\frac{2(2\nu+2n+1)(\nu+n)}{(m+n+2\nu)}C_{m-n-1}^{(\nu+n+1)}(x) +\frac{4(\nu+n)(\nu+n+1)(2\nu+2n+3)}{(m+n+2\nu)(m+n+2\nu+1)}xC_{m-n-2}^{(\nu+n+2)}(x).
\end{multline}
Applying \eqref{eq:Gegenbauer_lambda_p1} in the first and third term of \eqref{eq:nm_identities_L_2v2}, \eqref{eq:Gegenbauer_1mx2} 
in the second term of \eqref{eq:nm_identities_L_2v2} and \eqref{eq:Gegenbauer_three_term} in the fourth term of \eqref{eq:nm_identities_L_2v2}, we 
expand \eqref{eq:nm_identities_L_2v2} in terms of two Gegenbauer polynomials of parameter $\nu+n+2$ and of degree $m-n-1$ and $m-n-3$.
A straightforward computation shows that the coefficients of $C_{m-n-1}^{(\nu+n+2)}$ and $C_{m-n-3}^{(\nu+n+2)}$ are zero. 
\end{proof}

For the proof of Theorem \ref{thm:Pearson} we finally have to prove \eqref{eq:firstidentityforWPearson1}.
For this we rewrite \eqref{eq:firstidentityforWPearson1} as
\begin{equation}\label{eq:AnuL}
A^{(\nu)}(x)(L^{(\nu)}(x))^t - c^{(\nu)}(L^{(\nu)}(x))^t \, \Psi^{(\nu)}(x) = 0.
\end{equation}
We write the three-diagonal matrices $A^{(\nu)} $ and $\Psi^{(\nu)}$ as
\begin{align*}
A^{(\nu)}(x)&=\sum_k \, A_{k,k-1}E_{k,k-1} + \sum_k \, x\,A_{k,k}E_{k,k} +\sum_k\, (1-x^2)\,A_{k,k+1}\,E_{k,k+1},\\
c^{(\nu)}\Psi^{(\nu)}(x)&=\sum_k \, \psi_{k,k-1}E_{k,k-1} + \sum_k \, x\,\psi_{k,k}E_{k,k} +\sum_k\, \psi_{k,k+1}\,E_{k,k+1},
\end{align*}
where the explicit expression for the matrix entries follow from \eqref{eq:PhuPsiexplicit2}, \eqref{eq:Aexplicit}. 
The $(n,m)$-entry of the left hand side of \eqref{eq:AnuL} is given by
\begin{multline*}
% \label{eq:AnuL2}
\frac{n(m+n+2\nu-1)}{2(\nu+n-1)(2\nu+2n-1)} \, A_{n,n-1} C^{(\nu+n-1)}_{m-n+1}(x) +A_{n,n}\, x \, C^{(\nu+n)}_{m-n}(x) \\
+ \frac{2(\nu+n)(2\nu+2n+1)}{(m+n+2\nu)(n+1)} \, A_{n,n+1}\, (1-x^2)C^{(\nu+n+1)}_{m-n-1}(x) - \frac{(m+n+2\nu-1)}{m} \, \psi_{m-1,m}\, C^{(\nu+n)}_{m-n-1}(x)  \\
- \psi_{m,m}\, x\, C^{(\nu+n)}_{m-n}(x) - \frac{(m+1)}{(m+n+2\nu)} \, \psi_{m+1,m}\, C^{(\nu+n)}_{m-n+1}(x).
\end{multline*}
Now we apply \eqref{eq:Gegenbauer_lambda_p1} in the first term, the three-term recurrence relation \eqref{eq:Gegenbauer_three_term} in the second and 
fifth term and \eqref{eq:Gegenbauer_1mx2} in the third term in order to rewrite this expression in terms of just two 
Gegenbauer polynomials $C^{(\nu)}_{m-n-1}$ and $C^{(\nu)}_{m-n+1}$. 
Using the explicit expressions of $A^{(\nu)}$ and $\Psi^{(\nu)}$ 
it is a straightforward verification that the coefficients of $C^{(\nu)}_{m-n+1}$ and $C^{(\nu)}_{m-n-1}$  are zero. 
This completes the proof of \eqref{eq:firstidentityforWPearson}, and hence of Theorem \ref{thm:Pearson}.

%%%%%%%%%%%%%%%%%%%%%%%%%%%%%%%%%%%%%%%%%%%%%%%%%%%%%%%%%%%%%%%%%%%%%%
%%%%%%%%%%%%%%%%%%%%%%%%%%%%%%%%%%%%%%%%%%%%%%%%%%%%%%%%%%%%%%%%%%%%%%
%%%%%%%%%%%%%%%%%%%%%%%%%%%%%%%%%%%%%%%%%%%%%%%%%%%%%%%%%%%%%%%%%%%%%%

\end{document}